\documentclass[
11pt,                          
final,                         
english                        
]{article}

\usepackage[english]{babel}    
\usepackage{amsmath}           
\usepackage{amssymb}           
\usepackage[utf8]{inputenc}    
\usepackage[T1]{fontenc}       
\usepackage[a4paper]{geometry} 
\usepackage{setspace}          
\usepackage{bbm}               
\usepackage{mathrsfs}          
\usepackage{mathtools}         
\usepackage{stmaryrd}          
\usepackage[amsmath,thmmarks,hyperref]{ntheorem} 
\usepackage{exscale}           
\usepackage[sort]{cite}        
\usepackage{xspace}            
\usepackage[normalem]{ulem}    
\usepackage[activate={true,nocompatibility},
            final,tracking=true,kerning=true,
            spacing=true,factor=1100,
            stretch=10,shrink=10,expansion=false
           ]{microtype}
\microtypecontext{spacing=nonfrench}  
\usepackage[hyperref]{xcolor}
\definecolor{mydarkblue}{RGB}{0,0,155}
\usepackage[backref=page,      
           final=true,         
           colorlinks=true,    
           allcolors=mydarkblue,  
           hypertexnames=false,
           plainpages=false,              
           pdfpagelabels=true,            
           pdfencoding=auto,              
           unicode=true                   
                      ]{hyperref}         
\usepackage[shortcuts]{extdash}  
\usepackage{longtable}

\author{
  \textbf{Matthias Schötz}
  \thanks{Boursier de l'ULB, \href{mailto:Matthias.Schotz@ulb.ac.be}{\texttt{Matthias.Schotz@ulb.ac.be}}.
  This work was supported by the Fonds de la Recherche Scientifique (FNRS) and the Fonds Wetenschappelijk
  Onderzoek - Vlaaderen (FWO) under EOS Project n$^0$30950721.}\\
  Département de Mathématiques\\
  Université libre de Bruxelles
}

\makeatletter

\renewcommand{\mathbb}[1]{\mathbbm{#1}}

\newcommand{\refitem}[1] {\textit{\ref{#1}.)}}

\numberwithin{equation}{section}

\let\originalleft\left
\let\originalright\right
\renewcommand{\left}{\mathopen{}\mathclose\bgroup\originalleft}
\renewcommand{\right}{\aftergroup\egroup\originalright}

\newcommand{\lemmachairxname}{Lemma}
\newcommand{\propositionchairxname}{Proposition}
\newcommand{\theoremchairxname}{Theorem}
\newcommand{\corollarychairxname}{Corollary}
\newcommand{\definitionchairxname}{Definition}
\newcommand{\examplechairxname}{Example}
\newcommand{\proofchairxname}{Proof}

\StartBabelCommands{english}{extras}
\SetString{\lemmachairxname}{Lemma}
\SetString{\propositionchairxname}{Proposition}
\SetString{\theoremchairxname}{Theorem}
\SetString{\corollarychairxname}{Corollary}
\SetString{\definitionchairxname}{Definition}
\SetString{\examplechairxname}{Example}
\SetString{\proofchairxname}{Proof}
\EndBabelCommands

\theoremheaderfont{\normalfont\bfseries}
\theorembodyfont{\itshape}
\newtheorem{lemma}{\lemmachairxname}[section]

\newtheorem{proposition}[lemma]{\propositionchairxname}
\newtheorem{theorem}[lemma]{\theoremchairxname}
\newtheorem{corollary}[lemma]{\corollarychairxname}
\newtheorem{definition}[lemma]{\definitionchairxname}

\theorembodyfont{\rmfamily}

\newtheorem{example}[lemma]{\examplechairxname}

\def\theorem@checkbold{}

\theoremheaderfont{\normalfont\bfseries}
\theorembodyfont{\normalfont}
\theoremstyle{nonumberplain}
\theoremseparator{:}
\theoremsymbol{\hbox{$\boxempty$}}
\newtheorem{proof}{\proofchairxname}
\theoremsymbol{\hbox{$\triangledown$}}

\newcommand{\RE}             {\mathsf{Re}}
\newcommand{\IM}             {\mathsf{Im}}
\newcommand{\I}              {\mathrm{i}}

\newcommand{\cc}[1]          {\overline{{#1}}}
\newcommand{\Unit}           {\mathbb{1}}
\newcommand{\id}              {\mathrm{id}}
\newcommand{\argument}       {\ignorespaces{\,\cdot\,}\ignorespaces}
\DeclarePairedDelimiter{\abs}{\lvert}{\rvert}
\DeclarePairedDelimiter{\norm}{\lVert}{\rVert}
\newcommand{\Stetig}         {\mathscr{C}}
\DeclareMathOperator{\spec}    {\mathrm{spec}}

\DeclarePairedDelimiter{\ordinarySet}{\{}{\}}

\newcommand{\CC}{\mathbb{C}}
\newcommand{\RR}{\mathbb{R}}
\newcommand{\NN}{\mathbb{N}}
\renewcommand{\SS}{\mathbb{S}}
\newcommand{\set}[3][]{\ordinarySet[#1]{\,#2 \;#1|\; #3\,}}
\newcommand{\Hermitian}{\textup{H}}
\newcommand{\seminorm}[3][]{\norm[#1]{#3}_{#2}}

\newcommand{\metric}{\mathrm{d}}
\newcommand{\at}[2][]{#1|_{#2}}
\newcommand{\Sus}{\textit{Su}$^*$}

\newcommand{\bd}{\mathrm{bd}}

\newcommand{\pr}{\mathrm{pr}}
\newcommand{\neu}[1]{\uline{#1}}
\newcommand{\genCstar}[1]{\langle\!\langle\,#1\,\rangle\!\rangle_{C^*}}

\makeatother

\title{%
\texorpdfstring{%
Universal Continuous Calculus for \Sus\=/Algebras%
}{%
Universal Continuous Calculus for Su*-Algebras%
}%
}

\date{November 2020}

\begin{document}
\begin{onehalfspace}
\allowdisplaybreaks

\maketitle

\begin{abstract}
  Universal continuous calculi are defined and it is shown that for every finite tuple
  of pairwise commuting Hermitian elements of a \Sus\=/algebra
  (an ordered $^*$\=/algebra that is symmetric, i.e. ``strictly'' positive elements are 
  invertible, and uniformly complete), such a universal continuous calculus exists. This generalizes
  the continuous calculus for $C^*$\=/algebras to a class of generally unbounded ordered $^*$\=/algebras.
  On the way, some results about $^*$\=/algebras of continuous functions on locally 
  compact spaces are obtained. The approach used throughout is rather elementary
  and especially avoids any representation theory.
\end{abstract}




\section{Introduction}
\label{sec:Introduction}
The term $^*$\=/algebra will always refer to a unital associative (but not necessarily commutative)
$^*$\=/algebra over the scalar field of complex numbers. \Sus\=/algebras have been introduced
and examined in \cite{schoetz:EquivalenceOrderAlgebraicStructure}.
It was shown that they preserve many basic properties of $C^*$\=/algebras, but may have unbounded elements,
and that the commutative \Sus\=/algebras are complexifications of complete $\Phi$\=/algebras.
Essentially, they are $^*$\=/algebras endowed with a partial order on the Hermitian elements,
which are complete with respect to a metric topology induced by the order and fulfil one of
several equivalent additional conditions like existence of absolute values, of square roots of positive elements, 
of certain finite suprema or infima, or of inverses of Hermitian elements that are ``strictly'' positive.

Like for $C^*$\=/algebras, one can now ask whether these constructions can be generalized
to a well-defined way to apply more or less arbitrary continuous functions to algebra elements.
More precisely, given $N\in \NN$ and pairwise commuting Hermitian elements $a_1,\dots,a_N$ of a 
\Sus\=/algebra $\mathcal{A}$, does there
exist a unital $^*$\=/homomorphism $\Phi$ from a $^*$\=/algebra $\mathcal{I}$ of continuous 
complex-valued functions on 
$\RR^N$, or on a closed subset of $\RR^N$, to $\mathcal{A}$, that maps the $N$ coordinate functions
to $a_1,\dots,a_N$? If $N=1$ and $a\coloneqq a_1$, then $\Phi$ would map
polynomial functions to the corresponding polynomials of $a$. Similarly, one would expect
that e.g. the absolute value function $\abs{\argument} \colon \RR\to{[0,\infty[}$ is mapped to the
absolute value of $a$, or, if $a$ is positive, the square root function 
$\sqrt{\argument} \colon {[0,\infty[} \to {[0,\infty[}$ should be mapped to the square root of $a$.
Likewise, for $N=2$, the function ${\max} \colon \RR^2 \to \RR$, $(x_1,x_2) \mapsto \max \{ x_1,x_2 \}$
should be mapped to the supremum $a_1 \vee a_2$. As a continuous calculus necessarily
maps to a commutative unital $^*$\=/subalgebra of $\mathcal{A}$, the continuous calculus for
$\Phi$\=/algebras that has been constructed in \cite{buskes.pagter.vanRooij:FunctionalCalculusOnRieszSpaces}
already provides a solution for this in some cases.

Aside from the existence of such a unital $^*$\=/homomorphism $\Phi$, one would also
want $\Phi$ to be uniquely determined. Once the $^*$\=/algebra of functions
$\mathcal{I}$ is fixed, uniqueness of $\Phi$ can typically be obtained e.g. from the
generalized Stone-Weierstraß Theorem of \cite{schoetz:StoneWeierstrassTheoremsForRieszIdeals}.
But finding a good choice for $\mathcal{I}$ is not easy: Following
\cite[Thm.~4.12]{buskes.pagter.vanRooij:FunctionalCalculusOnRieszSpaces},
the algebra $\mathcal{I}$ can be chosen as the $^*$\=/algebra of all polynomially 
bounded continuous functions on $\RR^N$ and it was shown that for this choice, $\Phi$ is uniquely determined. 
Yet this way, neither the square root function $\sqrt{\argument} \colon {[0,\infty[} \to {[0,\infty[}$
nor the exponential function $\exp\colon \RR\to\RR$ are elements of $\mathcal{I}$ in the one-dimensional
case.
While one might not be able to construct continuous calculi for all Hermitian elements $a$
of $\mathcal{A}$ that allow to give $\sqrt{a}$ and $\exp(a)$ a well-defined meaning (e.g. if
$a$ is not positive in the first case, or if $a$ is ``not bounded enough'' in some sense to let $\exp(a)$
be an element of $\mathcal{A}$) there certainly are cases where this is possible. In
general, there might be several different continuous calculi for the same $N$\=/tuple
$a_1,\dots,a_N$, that differ in the choice of the algebra of functions on which
they are defined.

For $C^*$\=/algebras, there exists an easy natural choice for $\mathcal{I}$ at least in
the one-dimensional case: One defines $\mathcal{I}$ to be 
the $C^*$\=/algebra of all continuous functions on the compact spectrum of $a$.
This allows the application of all continuous functions with domain of definition
larger than $\spec(a)$, especially also the application of the square root function if
$a$ is positive, which is the case if and only if $\spec(a) \subseteq {[0,\infty[}$.
It will be shown that this natural choice is essentially also viable for \Sus\=/algebras, 
but there arises an additional problem: As elements of \Sus\=/algebras
may be unbounded, the spectrum $\spec(a)$ is no longer necessarily compact,
hence there typically are several different $^*$\=/algebras in between
the $^*$\=/algebra $\Stetig\big(\spec(a)\big)$ of all continuous functions and 
its $^*$\=/subalgebra of polynomially bounded ones.
The best choice will of course be the largest possible one.

This leads to the concept of a universal continuous calculus
in Definition~\ref{definition:univcalculi}:
There might be many continuous calculi for the same tuple $a_1,\dots,a_N$,
and the universal continuous calculus (if it exists) is the one through which all
others can be factored in a natural way. Consequently, the universal continuous
calculus is the most general one, and the algebra of functions on which it is
defined carries information about the elements $a_1,\dots,a_N$.
The main result, Theorem~\ref{theorem:univcc}, establishes the existence of such a 
universal continuous calculus for all $N$\=/tuples of pairwise commuting Hermitian 
elements of a \Sus\=/algebra. This will be proven by a rather explicit 
construction using elementary methods. Moreover, it will be shown that the universal
continuous calculus is an order embedding of a \Sus\=/algebra of continuous functions
into a general \Sus\=/algebra.
This allows to prove identities or estimates involving finitely
many pairwise commuting Hermitian elements and their absolute values, square roots, etc.,
by discussing only the special and easy case that they are given by continuous functions.
It also yields a representation of certain \Sus\=/algebras for which a surjective
continuous calculus exists.
This is roughly analogous to the result provided by \cite{buskes.vanRooij:SmallRieszSpaces}
for Riesz spaces. Note, however, that the approach taken here is quite different:
While \cite{buskes.pagter.vanRooij:FunctionalCalculusOnRieszSpaces} applies the
representation theorem obtained in \cite{buskes.vanRooij:SmallRieszSpaces} for
``small'' Riesz spaces in order to derive a functional calculus, the functional
calculus for \Sus\=/algebras here will be obtained by extending the polynomial
calculus and can then be used to obtain a representation theorem for ``small'' commutative \Sus\=/algebras.

The article is organized as follows: After recapitulating some preliminaries and fixing the
notation in the next Section~\ref{sec:preliminaries}, Section~\ref{sec:ccdef} discusses
the fundamental definitions that are relevant for this article, especially (universal)
continuous calculi and commutative \Sus\=/algebras of finite type.
In Section~\ref{sec:propersus}, certain well-behaved commutative \Sus\=/algebras of
continuous functions will be examined in some detail.
The construction of universal continuous calculi is presented in Section~\ref{sec:cccon}
and results in the main Theorem~\ref{theorem:univcc}. An application of this is a representation
theorem for commutative \Sus\=/algebras of finite type, which is discussed in the final 
Section~\ref{sec:rep} and which gives an algebraic characterization of those ordered $^*$\=/algebras
of continuous functions on a closed subset of $\RR^N$ that contain all uniformly bounded 
continuous functions and at least one proper function.
\section{Preliminaries} \label{sec:preliminaries}
The notation essentially follows~\cite{schoetz:EquivalenceOrderAlgebraicStructure},
where the fundamental properties of ordered $^*$\=/algebras
and especially \Sus\=/algebras have been discussed in more detail and more generality.
Natural numbers are $\NN = \{1,2,3,\dots\}$ and $\NN_0 \coloneqq \NN\cup\{0\}$, and 
the fields of real and complex numbers are denoted by $\RR$ and $\CC$, respectively.
Let $X$ be a set, then $\id_X\colon X\to X$ is $x\mapsto \id_X(x):= x$. A \neu{partial order}
on $X$ is a reflexive, transitive and anti-symmetric relation. If $X$ and $Y$ are both 
endowed with a partial order, then a map $\Phi \colon X \to Y$ is called \neu{increasing}
if $\Phi(x) \le \Phi(x')$ for all $x,x'\in X$ with $x\le x'$. If $\Phi$ is
injective, increasing and also fulfils $x \le x'$ for all $x,x'\in X$ with $\Phi(x) \le \Phi(x')$,
then $\Phi$ is called an \neu{order embedding}.

A \neu{$^*$\=/algebra} is a unital associative algebra over the field $\CC$ that
is endowed with an anti-linear involution $\argument^*$ fulfilling $(ab)^* = b^*a^*$
for all its elements $a$ and $b$. 
The unit of a $^*$\=/algebra $\mathcal{A}$ is denoted by $\Unit$, or, more explicitly,
by $\Unit_\mathcal{A}$, and automatically fulfils $\Unit^* = \Unit$. Note that it is not
required that $0 \neq \Unit$, but the only case where $0 = \Unit$ is 
the trivial $^*$\=/algebra $\{0\}$. An element $a\in \mathcal{A}$ is called \neu{Hermitian}
if $a=a^*$, and $\mathcal{A}_\Hermitian \coloneqq \set{a\in \mathcal{A}}{a=a^*}$ denotes the
real linear subspace of Hermitian elements of $\mathcal{A}$.
Every $a\in \mathcal{A}$ can be decomposed as $a = \RE(a) + \I\, \IM(a)$ with uniquely determined
Hermitian real and imaginary parts of $a$, which are explicitly given by $\RE(a) = \frac{1}{2}(a+a^*)$ and $\IM(a) = \frac{1}{2\I}(a-a^*)$.
Moreover, if $S\subseteq \mathcal{A}$ is stable under the $^*$\=/involution, then its \neu{commutant}
$S' \coloneqq \set[\big]{a\in \mathcal{A}}{\forall_{s\in S}:as=sa}$
and its \neu{bicommutant} $S''$ are \neu{unital $^*$\=/subalgebras} of $\mathcal{A}$, i.e.
unital subalgebras that are stable under the $^*$\=/involution and thus are $^*$\=/algebras again.
Note that $S''$ is commutative if all elements in $S$ are pairwise commuting.
The unital $^*$\=/subalgebra \neu{generated} by a subset $S$ of $\mathcal{A}$ is the smallest (with respect to $\subseteq$)
unital $^*$\=/subalgebra of $\mathcal{A}$ that contains $S$, and can explicitly be constructed as the set
of all finite sums of finite products of elements of $S \cup \set{s^*}{s\in S} \cup \set{\lambda \Unit}{\lambda \in \CC}$.
A linear subspace $\mathcal{I}$ of a $^*$\=/algebra $\mathcal{A}$ is called a \neu{$^*$\=/ideal} if it is stable
under the $^*$\=/involution and fulfils $ab \in \mathcal{I}$ for all $a\in \mathcal{A}$ and all $b\in \mathcal{I}$,
and then automatically also fulfils $ba \in \mathcal{I}$ for all $a\in \mathcal{A}$ and all $b\in \mathcal{I}$.
Given two $^*$\=/algebras $\mathcal{A}$ and $\mathcal{B}$, then a map $\Phi \colon \mathcal{A} \to \mathcal{B}$
is a \neu{unital $^*$\=/homomorphism} if it is linear, maps $\Unit_\mathcal{A}$ to $\Unit_B$ and fulfils $\Phi(a\tilde{a}) = \Phi(a) \Phi(\tilde{a})$
and $\Phi(a^*) = \Phi(a)^*$ for all $a, \tilde{a}\in \mathcal{A}$. Its kernel $\ker \Phi \coloneqq \set{a\in\mathcal{A}}{\Phi(a)=0}$
automatically is a $^*$\=/ideal. As an example, $\CC[t_1,\dots,t_N]$ denotes the $^*$\=/algebra of polynomials in $N\in \NN$ Hermitian arguments,
i.e.~its $^*$\=/involution is given by complex conjugation of the coefficients so that $\CC[t_1,\dots,t_N]_\Hermitian \cong \RR[t_1,\dots,t_N]$. This is the 
$^*$\=/algebra that is freely generated by $N$ pairwise commuting Hermitian elements $t_1,\dots,t_N$,
so for every $^*$\=/algebra $\mathcal{A}$ and every $N$-tuple of pairwise commuting $a_1,\dots,a_N \in \mathcal{A}_\Hermitian$
there exists a unique unital $^*$\=/homomorphism from $\CC[t_1,\dots,t_N]$ to $\mathcal{A}$ that maps
$t_n$ to $a_n$ for all $n\in \{1,\dots,N\}$. This unital $^*$\=/homomorphism is denoted as usual as evaluating a
polynomial on $a_1,\dots,a_N$, i.e.~as $\CC[t_1,\dots,t_N] \ni \pi \mapsto \pi(a_1,\dots,a_N) \in \mathcal{A}$.
The idea of a functional calculus is to extend this ``polynomial calculus'' to more interesting $^*$\=/algebras like $\Stetig(\RR^N)$,
the continuous complex valued functions on $\RR^N$ with the pointwise operations.

An \neu{ordered $^*$\=/algebras} is a $^*$\=/algebra $\mathcal{A}$ that carries a partial order $\le$ on $\mathcal{A}_\Hermitian$ such that
\begin{equation}
  a+c\le b+c\,,\quad\quad
  d^*a\,d \le d^*b\,d \quad\quad\text{and}\quad\quad
  0 \le \Unit \label{eq:orderedstaralg}
\end{equation}
hold for all $a,b,c\in \mathcal{A}_\Hermitian$ with $a\le b$ and all $d\in \mathcal{A}$. One writes $\mathcal{A}_\Hermitian^+ \coloneqq \set{a\in \mathcal{A}_\Hermitian^+}{a\ge 0}$
for the convex cone of \neu{positive} Hermitian elements of $\mathcal{A}$, and it is not hard to check that the order $\le$ on $\mathcal{A}_\Hermitian$
is completely determined by $\mathcal{A}_\Hermitian^+$. As $\Unit \in \mathcal{A}_\Hermitian^+$, $a + b \in \mathcal{A}_\Hermitian^+$ for all $a,b\in \mathcal{A}_\Hermitian^+$
and $c^*a\,c \in \mathcal{A}_\Hermitian$ for all $a\in \mathcal{A}_\Hermitian^+$ and all $c\in \mathcal{A}$, this set $\mathcal{A}_\Hermitian^+$ is a \neu{quadratic module}.
A unital $^*$\=/subalgebra $\mathcal{B}$ of an ordered $^*$\=/algebra $\mathcal{A}$ will always be endowed with the order that $\mathcal{B}_\Hermitian$
inherits from $\mathcal{A}_\Hermitian$.
If $\mathcal{A}$ and $\mathcal{B}$ are two ordered $^*$\=/algebras and $\Phi \colon \mathcal{A} \to \mathcal{B}$
a unital $^*$\=/homomorphism, then $\Phi$ is called \neu{positive} if its restriction to a map from $\mathcal{A}_\Hermitian$ to $\mathcal{B}_\Hermitian$
is increasing, or equivalently, if $\Phi(a) \in \mathcal{B}^+_\Hermitian$ for all $a\in \mathcal{A}^+_\Hermitian$.

For example, the $^*$\=/algebra $\CC[t_1,\dots,t_N]$ becomes an ordered $^*$\=/algebra by endowing $\CC[t_1,\dots,t_N]_\Hermitian$ with the pointwise order, 
which is the order for which $\CC[t_1,\dots,t_N]_\Hermitian^+$
is the set of all $\pi \in \CC[t_1,\dots,t_N]_\Hermitian$ that fulfil $\pi(x_1,\dots,x_N) \ge 0$ for all $x\in \RR^N$.
One might ask now whether the canonical unital $^*$\=/homomorphism $\CC[t_1,\dots,t_N] \ni \pi \mapsto \pi(a_1,\dots,a_N) \in \mathcal{A}$
is positive for every choice of pairwise commuting Hermitian elements $a_1,\dots,a_N$ of any ordered $^*$\=/algebra $\mathcal{A}$
and with respect to the pointwise order on $\CC[t_1,\dots,t_N]$. While this can be shown to be true in the one-dimensional case
by expressing every $\pi \in \CC[t]^+_\Hermitian$ as a sum of polynomials of the form $\rho^* \rho$ with $\rho\in \CC[t]$
(which is possible as a consequence of the fundamental theorem of algebra), it does no longer hold if $N\ge 2$:
Indeed, the identity map $\id_{\CC[t_1,\dots,t_N]}$
is not positive as a unital $^*$\=/homomorphism from $\CC[t_1,\dots,t_N]$ with the pointwise order
to $\CC[t_1,\dots,t_N]$ with the order for which only elements of the form $\sum_{k=1}^K \rho_k^*\rho_k$ with $K \in \NN$ and 
$\rho_1,\dots,\rho_K \in \CC[t_1,\dots,t_N]$ are positive: For example, the famous Motzkin polynomial
$t_1^4 t_2^2 + t_1^2t_2^4 - 3t_1^2t_2^2 + 1$ is pointwise positive but cannot be expressed as a sum of squares.
This already indicates that the theory of quadratic modules on polynomial algebras in real algebraic geometry
is highly non-trivial. Important results are the Positivstellensatz by Krivine and Stengle,
\cite{krivine:positivstellensatz} and \cite{stengle:positivstellensatz}, which gives an algebraic
description of polynomials that are pointwise positive or strictly positive on a semi-algebraic set,
and its variants in compact cases by Handelman \cite{handelman:representingPolynomialsByPositiveLinearFunctionsOnCompactConvexPolyhedra},
Schmüdgen \cite{schmuedgen:KMomentProblemForCompactSemiAlgebraicSets}, and Putinar \cite{putinar:positivstellensatz}.
At least in some special cases, such algebraic descriptions can also be obtained constructively by
exploiting a classical result of P{\'o}lya, see e.g.~\cite{averkov:ConstructiveProofsOfSomePositivstellensaetze}.
In a similar way, \cite{habicht:ZerlegungStriktDefiniterFormen} provides a constructive solution
to Hilbert's 17th problem in the special case of strictly positive polynomials. These constructive
methods will be helpful later on in the proof of Lemma~\ref{lemma:rationalcalculus} for
extending the polynomial calculus to a continuous calculus without having to resort to the full might
of the Positivstellensatz.

Another important example of ordered $^*$\=/algebras is $\Stetig(X)$, the $^*$\=/algebra
of complex-valued continuous functions on a topological space $X$ with the 
pointwise operations and pointwise comparison (if $X=\emptyset$, then $\Stetig(X) \cong \{0\}$). Similarly, every $C^*$\=/algebra
becomes an ordered $^*$\=/algebra in a natural way by declaring those Hermitian elements
to be positive whose spectrum is a subset of ${[0,\infty[}$. These are two
classes of examples of \Sus\=/algebras, which will be discussed 
in the following (see \cite{schoetz:EquivalenceOrderAlgebraicStructure} for details):

An ordered $^*$\=/algebra $\mathcal{A}$ is called \neu{Archimedean} if $\mathcal{A}_\Hermitian$ is an Archimedean
ordered vector space, i.e.~if the following condition is fulfilled: Whenever $a \le \epsilon b$
holds for one $a \in \mathcal{A}_\Hermitian$, one $b\in \mathcal{A}_\Hermitian^+$ and all $\epsilon \in {]0,\infty[}$,
then $a \le 0$. It is important to point out that this is not related to the notion of an Archimedean
quadratic module in real algebraic geometry.
Using the convention that $-\infty \Unit \le a \le \infty \Unit$ and $a \le \infty^2 \Unit$ are true for
all $a\in \mathcal{A}_\Hermitian$, one can define on every Archimedean ordered $^*$\=/algebra $\mathcal{A}$ the map 
$\seminorm{\infty}{\argument} \colon \mathcal{A} \to [0,\infty]$,
\begin{equation}
  a \mapsto \seminorm{\infty}{a} \coloneqq \min \set[\big]{\lambda \in {[0,\infty]}}{a^*a \le \lambda^2 \Unit}
  \label{eq:inftyseminorm}
\end{equation}
and the subset $\mathcal{A}^\bd\coloneqq \set[\big]{a\in\mathcal{A}}{\seminorm{\infty}{a}<\infty}$ of \neu{uniformly bounded} elements of $\mathcal{A}$. 
Alternatively,
\begin{equation}
  \seminorm{\infty}{a} = \min \set[\big]{\lambda \in {[0,\infty]}}{{-\lambda} \Unit \le a \le \lambda \Unit}
  \label{eq:inftyseminormalt}
\end{equation}
holds for all $a\in \mathcal{A}_\Hermitian$. Then $\mathcal{A}^\bd$ is a unital $^*$\=/subalgebra of $\mathcal{A}$ on which the restriction
of $\seminorm{\infty}{\argument}$ is a \neu{$C^*$\=/norm}, i.e.~a norm
fulfilling $\seminorm{\infty}{ab} \le \seminorm{\infty}{a} \seminorm{\infty}{b}$ and $\seminorm{\infty}{a^*a} = \seminorm{\infty}{a}^2$
for all $a,b\in \mathcal{A}^\bd$. This relation between Archimedean ordered $^*$\=/algebras and $C^*$\=/algebras
has been described first in \cite{cimpric:representationTheoremForArchimedeanQuadraticModules},
and might serve as a motivation to study Archimedean ordered $^*$\=/algebras as abstractions
of $^*$\=/algebras of possibly unbounded operators on a Hilbert space, i.e.~of $O^*$\=/algebras as in \cite{schmuedgen:UnboundedOperatorAlgebraAndRepresentationTheory}.
In the unbounded case,
if $\mathcal{A}^\bd \neq \mathcal{A}$, the map $\seminorm{\infty}{\argument}$ does not describe a norm on $\mathcal{A}$
but can still be used to construct a translation-invariant
metric $\metric_\infty$ on $\mathcal{A}$, called the \neu{uniform metric}, as
\begin{equation}
  (a,b) \mapsto \metric_\infty(a,b) \coloneqq \min\big\{ \seminorm{\infty}{a-b}, 1 \big\}
\end{equation}
for all $a,b\in \mathcal{A}$. An Archimedean ordered $^*$\=/algebra is called \neu{uniformly complete} if it
is complete with respect to $\metric_\infty$. All topological and metric notions for
Archimedean ordered $^*$\=/algebras will always refer to this uniform metric. For example,
every positive unital $^*$\=/homomorphism between two Archimedean ordered $^*$\=/algebras is
automatically continuous.

If $\mathcal{A}$ is an ordered $^*$\=/algebra, then an element $a\in \mathcal{A}_\Hermitian$ is called \neu{coercive}
if there exists an $\epsilon \in {]0,\infty[}$ such that $\epsilon \Unit \le a$, 
thus coercive elements are especially positive. If the multiplicative inverse $a^{-1}$ of this coercive $a$ exists in $\mathcal{A}$, then 
it is Hermitian and $0 \le a^{-1} \le \epsilon^{-1} \Unit$.
An ordered $^*$\=/algebra in which all coercive elements are invertible
is called \neu{symmetric}. A \neu{\Sus-algebra} finally is a symmetric and uniformly complete Archimedean ordered $^*$\=/algebra.

Examples are all $C^*$\=/algebras with the canonical order (these are the \Sus\=/algebras $\mathcal{A}$ for which $\mathcal{A} = \mathcal{A}^\bd$)
as well as complexifications of complete $\Phi$\=/algebras (which are the commutative \Sus\=/algebras),
especially $\Stetig(X)$ for every topological space $X$.
See e.g. \cite{henriksen.johnson:structureOfArchimedeanLatticeOrderedAlgebras} for more details
on $\Phi$\=/algebras. One can check that $\seminorm{\infty}{f} = \sup_{x\in X} \abs{f(x)}$
holds for all $f\in \Stetig(X)$ and all topological spaces $X$, and it is worth mentioning
that the order on $\Stetig(X)_\Hermitian$ is usually easier to work with than the axioms \eqref{eq:orderedstaralg}.
Given a continuous map $\phi \colon X \to Y$ between two topological spaces $X$ and $Y$,
then one obtains a positive unital $^*$\=/homomorphism $\Stetig(Y) \ni f \mapsto f \circ \phi \in \Stetig(X)$.
In the special case that $X$ is a subset of $Y$ and $\phi$ the inclusion map, then $f\circ \phi$
is simply the restriction of $f\in\Stetig(Y)$ to $X$, denoted by $f\at{X} \in \Stetig(X)$.
More examples of \Sus\=/algebras can be constructed as $^*$\=/algebras of unbounded operators on a Hilbert space
if some selfadjointness conditions are fulfilled that guarantee the existence
of inverses of coercive elements. \Sus\=/algebras have some nice properties, for example, one can
construct square roots of positive Hermitian elements, or absolute values of Hermitian elements,
and also adapted versions of infima and suprema of finitely many commuting Hermitian elements.
However, for the purpose of this article, these constructions are only needed in the special case
of $\Stetig(X)$, where they simply describe the pointwise square root, pointwise absolute value, and pointwise 
minimum or maximum of real-valued functions. The universal continuous calculus then provides
an equivalent way to generalize these constructions to arbitrary \Sus\=/algebras.
It is also noteworthy that on a \Sus\=/algebra $\mathcal{A}$, in contrast to the case of ordered $^*$\=/algebras
of polynomials, the order is uniquely determined in the sense that there exists only one partial
order on the Hermitian elements of the underlying $^*$\=/algebra that fulfils the axioms \eqref{eq:orderedstaralg}.
\section{Universal Continuous Calculus -- Definitions} \label{sec:ccdef}
Consider a \Sus\=/algebra $\mathcal{A}$. Then every unital $^*$\=/subalgebra $\mathcal{B}$ of $\mathcal{A}$
with the order inherited from $\mathcal{A}$ is again an Archimedean ordered $^*$\=/algebra. It is
even a \Sus\=/algebra if and only if it is symmetric itself and closed with respect to the uniform
metric $\metric_\infty$. One obvious example for this is $\mathcal{B} = \mathcal{A}^\bd$, the $C^*$-algebra of all 
uniformly bounded elements of $\mathcal{A}$. However, there arise some difficulties
when one attempts to construct other examples of closed symmetric unital $^*$\=/subalgebras: The naive approach
to start with a symmetric unital $^*$\=/subalgebra and take its closure with respect to $\metric_\infty$
might fail because left and right multiplication with elements of $\mathcal{A}$ is in general not continuous with 
respect to $\metric_\infty$, which makes it harder to guarantee that the closure is again a subalgebra. 
Nevertheless, there are important special cases which are easier to understand: 
\begin{definition}
  Let $\mathcal{A}$ be an ordered $^*$\=/algebra. Then an \neu{intermediate $^*$\=/subalgebra of $\mathcal{A}$}
  is a (necessarily unital) $^*$\=/subalgebra $\mathcal{I}$ of $\mathcal{A}$ fulfilling 
  $\mathcal{A}^\bd \subseteq \mathcal{I}$.
\end{definition}
\begin{proposition} \label{proposition:intermediateSus}
  Let $\mathcal{A}$ be a \Sus\=/algebra and $\mathcal{I}$ an intermediate $^*$\=/subalgebra
  of $\mathcal{A}$, then $\mathcal{I}$ is a closed unital $^*$\=/subalgebra of $\mathcal{A}$ and symmetric,
  hence is  again a \Sus\=/algebra. Moreover, whenever
  $\ell,u \in \mathcal{I}_\Hermitian$ and $a \in \mathcal{A}_\Hermitian$ fulfil $\ell \le a \le u$,
  then also $a\in \mathcal{I}_\Hermitian$.
\end{proposition}
\begin{proof}
  As the inverse of every coercive element in $\mathcal{I}_\Hermitian$ is uniformly bounded, 
  $\mathcal{I}$ is again symmetric. 
  Moreover, if $a\in \mathcal{A}_\Hermitian$ and $\ell,u \in \mathcal{I}_\Hermitian$
  fulfil $\ell \le a \le u$, define 
  $b \coloneqq \frac{1}{2}(\Unit - \ell)^2 + \frac{1}{2}(\Unit + u)^2 + \Unit \in \mathcal{I}_\Hermitian^+$.
  Then $-b \le -\frac{1}{2}(\Unit - \ell)^2 \le \ell$ and $u \le \frac{1}{2}(\Unit + u)^2 \le b$
  imply $-b \le a \le b$. As $b$ is coercive with $b\ge \Unit$ it follows that 
  $b^{-1} \in (\mathcal{A}^\bd)_\Hermitian^+$ exists and that
  $-b^{-1} \le b^{-1} a \,b^{-1} \le b^{-1}$.
  So $b^{-1} a \,b^{-1} \in (\mathcal{A}^\bd)_\Hermitian \subseteq \mathcal{I}_\Hermitian$ 
  and therefore also $a \in \mathcal{I}_\Hermitian$.
  
  It only remains to show that $\mathcal{I}$ is closed in $\mathcal{A}$, and as the $\RR$-linear projections
  $\mathcal{A} \ni a \mapsto \RE(a) = (a+a^*)/2 \in \mathcal{A}_\Hermitian$ and $\mathcal{A} \ni a \mapsto \IM(a) = (a-a^*)/(2\I) \in \mathcal{A}_\Hermitian$
  onto the real and imaginary parts are continuous by continuity of addition, $^*$\=/involution and scalar multiplication,
  it suffices to check that $\mathcal{I}_\Hermitian$ is closed in $\mathcal{A}_\Hermitian$: If a sequence  $(a_n)_{n\in \NN}$ 
  in $\mathcal{I}_\Hermitian$ converges against a limit $\hat{a} \in \mathcal{A}_\Hermitian$, then there exists an $n\in \NN$
  such that $-\Unit \le \hat{a}-a_n \le \Unit$, i.e. $a_n-\Unit \le \hat{a} \le a_n + \Unit$,
  which implies $\hat{a} \in \mathcal{I}_\Hermitian$.
\end{proof}

A continuous calculus should especially generalize the polynomial calculus. The continuous functions on $\RR^N$
that correspond to the generators $t_1,\dots,t_N$ of the $^*$\=/algebra $\CC[t_1,\dots,t_N]$ are the coordinate
functions:
\begin{definition}
  If $X$ is a closed subset of $\RR^N$ with $N\in \NN$, then the \neu{coordinate functions} on $X$
  are defined as $\pr_{X;n} \colon X \to \RR$, $(x_1,\dots,x_N) \mapsto \pr_{X;n}(x_1,\dots,x_N) \coloneqq x_n$
  for all $n\in \{1,\dots,N\}$. Oftentimes the set $X$ will be clear from the context,
  and then one simply writes $\pr_n$ instead of $\pr_{X;n}$.
\end{definition}
It is clear that these coordinate functions are continuous. We can now define continuous calculi:
\begin{definition} \label{definition:calculi}
  Let $\mathcal{A}$ be an ordered $^*$\=/algebra, $N\in \NN$ and $a_1,\dots,a_N\in \mathcal{A}_\Hermitian$,
  then a \neu{continuous calculus for $a_1,\dots,a_N$} is a triple $(X,\mathcal{I},\Phi)$ 
  of a closed subset $X$ of $\RR^N$, an intermediate $^*$\=/subalgebra $\mathcal{I}$
  of $\Stetig(X)$ with $\pr_{n} \in \mathcal{I}$ for all $n\in \{1,\dots,N\}$, and 
  a positive unital $^*$\=/homomorphism $\Phi \colon \mathcal{I} \to \mathcal{A}$ fulfilling $\Phi(\pr_{n}) = a_n$
  for all $n\in \{1,\dots,N\}$.
\end{definition}
Even though it is not explicitly required in the definition, such a continuous calculus for
an $N$-tuple of Hermitian elements $a_1,\dots,a_N$ can only exist if $a_1,\dots,a_N$ are
pairwise commuting, because the coordinate functions $\pr_{1},\dots,\pr_{N}$ are pairwise commuting.

Neither existence nor uniqueness of continuous calculi for fixed algebra elements $a_1,\dots,a_N$
are clear. It is especially possible that there are many different continuous calculi on different domains.
One would like to have a most general one:
\begin{definition} \label{definition:univcalculi}
  Let $\mathcal{A}$ be an ordered $^*$\=/algebra, $N\in\NN$ and $a_1,\dots,a_N\in \mathcal{A}_\Hermitian$,
  then a continuous calculus $(X,\mathcal{I},\Phi)$ for $a_1,\dots,a_N$ is called \neu{universal}
  if every continuous calculus $(Y,\mathcal{J},\Psi)$ for $a_1,\dots,a_N$ factors through 
  $(X,\mathcal{I},\Phi)$ in the following sense: $X \subseteq Y$ and for every $f\in \mathcal{J}$, its restriction 
  $f\at{X}$ to a function on $X$ is an element of $\mathcal{I}$ and $\Phi(f\at{X}) = \Psi(f)$ holds.
\end{definition}
It is immediately clear that the universal continuous calculus is uniquely determined (if it exists):
\begin{proposition}
  Let $\mathcal{A}$ be an ordered $^*$\=/algebra, $N\in \NN$ and $a_1,\dots,a_N\in \mathcal{A}_\Hermitian$,
  and let $(X,\mathcal{I},\Phi)$ and $(Y,\mathcal{J},\Psi)$ be two universal continuous calculi 
  for $a_1,\dots,a_N$, then $X=Y$, $\mathcal{I}=\mathcal{J}$ and $\Phi = \Psi$.
\end{proposition}
It thus makes sense to define:
\begin{definition}
  Let $\mathcal{A}$ be an ordered $^*$\=/algebra, $N\in \NN$ and $a_1,\dots,a_N\in \mathcal{A}_\Hermitian$
  such that the universal continuous calculus $(X,\mathcal{I},\Phi)$ exists.
  Then $X$ is called the \neu{spectrum of $a_1,\dots,a_N$} and is denoted by
  $\spec(a_1,\dots,a_N) \coloneqq X \subseteq \RR^N$. Similarly, define
  $\mathcal{F}(a_1,\dots,a_N) \coloneqq \mathcal{I} \subseteq \Stetig\big( \spec(a_1,\dots,a_N) \big)$
  and $\Gamma_{a_1,\dots,a_N} \coloneqq \Phi \colon \mathcal{F}(a_1,\dots,a_N) \to \mathcal{A}$.
\end{definition}

If $(X,\mathcal{I},\Phi)$ is a continuous calculus for a tuple $a_1,\dots,a_N$, $N\in \NN$, of Hermitian elements
of an ordered $^*$\=/algebra $\mathcal{A}$, then one can in general not expect $\mathcal{I}$ to be finitely generated as a $^*$\=/algebra.
However, it will become clear from the construction of continuous calculi that $\mathcal{I}$ still fulfils a
weaker finiteness condition that will be discussed in the remainder of this section:
\begin{definition} \label{definition:properelement}
  Let $\mathcal{A}$ be an ordered $^*$\=/algebra and $p\in \mathcal{A}_\Hermitian^+$, then $p$ is said to be \neu{proper in $\mathcal{A}$}
  if for every $\lambda \in {]0,\infty[}$ and every $a\in \mathcal{A}^+_\Hermitian$ there exist $\mu \in {[0,\infty[}$ and $b\in \mathcal{A}^+_\Hermitian$
  such that for every $\epsilon \in {]0,\infty[}$ one finds $k\in \NN_0$ for which the estimate $a \le \mu \Unit + (p/\lambda)^k + \epsilon b$
  holds.
\end{definition}
For example, if $\mathcal{A}$ is an Archimedean ordered $^*$\=/algebra in which every element is uniformly bounded, i.e.~$\mathcal{A}^\bd = \mathcal{A}$,
then every $p \in \mathcal{A}^+_\Hermitian$ is proper in $\mathcal{A}$ because $a \le \seminorm{\infty}{a} \Unit$ for all $a\in \mathcal{A}^+_\Hermitian$.
Less trivial examples will be discussed in the next Section~\ref{sec:propersus}.
\begin{definition}
  Let $\mathcal{A}$ be a \Sus\=/algebra and $S\subseteq \mathcal{A}^\bd$, then $\genCstar{S}$ denotes the \neu{$C^*$\=/subalgebra
  of $\mathcal{A}^\bd$ that is generated by $S$}, i.e.~the closure in $\mathcal{A}^\bd$ with respect to the $C^*$\=/norm $\seminorm{\infty}{\argument}$
  of the unital $^*$\=/subalgebra that is generated by $S$.
\end{definition}
As the restriction of $\seminorm{\infty}{\argument}$ to the uniformly bounded elements is a $C^*$\=/norm,
$\genCstar{S}$ is indeed a $C^*$\=/algebra.
\begin{lemma}
  Let $\mathcal{A}$ be a commutative \Sus\=/algebra, $N\in \NN$ and $a_1,\dots,a_N \in \mathcal{A}_\Hermitian$.
  Then one has $a_n (\Unit+a_1^2 + \dots + a_N^2)^{-1} \in \mathcal{A}^\bd$ for all $n\in \{1,\dots,N\}$.
\end{lemma}
\begin{proof}
  From $\Unit+a_1^2 + \dots + a_N^2 \ge \Unit$ it follows that $0 \le (\Unit+a_1^2 + \dots + a_N^2)^{-1} \le \Unit$,
  and therefore
  \begin{align*}
    a_n^2 (\Unit+a_1^2 + \dots + a_N^2)^{-2}
    \le
    (\Unit+a_1^2 + \dots + a_N^2)  (\Unit+a_1^2 + \dots + a_N^2)^{-2}
    =
    (\Unit+a_1^2 + \dots + a_N^2)^{-1}
  \end{align*}
  shows that $a_n (\Unit+a_1^2 + \dots + a_N^2)^{-1} \in \mathcal{A}^\bd$ for all $n\in \{1,\dots,N\}$.
\end{proof}
Because of this, the next definition makes sense:
\begin{definition}
  A commutative \Sus\=/algebra $\mathcal{A}$ is said to be \neu{of finite type} if there are $N\in \NN$ and
  $a_1,\dots,a_N \in \mathcal{A}_\Hermitian$ for which $p \coloneqq \Unit + a_1^2 + \dots + a_N^2$
  is proper in $\mathcal{A}$ and which generate $\mathcal{A}$ in the sense that for every $b \in \mathcal{A}$ one finds 
  numerator and denominator $b_\mathrm{n},b_\mathrm{d} \in \genCstar{\{p^{-1}, a_1 p^{-1}, \dots, a_N p^{-1}\}}$
  such that $b_\mathrm{d}$ is invertible in $\mathcal{A}$ and $b = b_\mathrm{n} b_\mathrm{d}^{-1}$.
  In this case $a_1,\dots,a_N$ are called \neu{generators} of $\mathcal{A}$.
\end{definition}
It should not come as a surprise that the surjective image of a commutative \Sus\=/algebra of finite type is again of finite type:
\begin{proposition} \label{proposition:surjectivefinite}
  Let $\mathcal{A}$ and $\mathcal{B}$ be two commutative \Sus\=/algebras and $\Phi \colon \mathcal{A} \to \mathcal{B}$
  a surjective positive unital $^*$\=/homomorphism. If $\mathcal{A}$ is of finite type with generators
  $a_1,\dots,a_N\in\mathcal{A}_\Hermitian$, $N\in \NN$, then $\mathcal{B}$ is also of finite type with
  generators $\Phi(a_1),\dots,\Phi(a_N)$.
\end{proposition}
\begin{proof}
  Write $p \coloneqq \Unit_\mathcal{A} + a_1^2 + \dots + a_N^2$, then $\Phi(p) = \Unit_B + \Phi(a_1)^2 + \dots + \Phi(a_N)^2$ is proper:
  Indeed, for every $\lambda \in {]0,\infty[}$ and every element of $\mathcal{B}^+_\Hermitian$, 
  which can be expressed as $\Phi(a)$ with $a \in \mathcal{A}_\Hermitian$ by taking the real part
  of any preimage under $\Phi$, there exist $\mu \in {[0,\infty[}$ and $\tilde{a} \in \mathcal{A}^+_\Hermitian$
  such that for every $\epsilon \in {]0,\infty[}$ one finds $k\in \NN_0$ for which the estimate
  $\frac{1}{2}(\Unit_\mathcal{A} + a^2) \le \mu \Unit_\mathcal{A} + (p / \lambda)^k + \epsilon \tilde{a}$ holds because $p$ is proper in $\mathcal{A}$.
  As a consequence, $\Phi(a) \le \frac{1}{2}(\Unit_\mathcal{B} + \Phi(a)^2) = \Phi\big( \frac{1}{2}(\Unit_\mathcal{A} + a^2) \big) \le \mu \Unit_\mathcal{B} + (\Phi(p) / \lambda)^k + \epsilon \Phi(\tilde{a})$ holds with $\Phi(\tilde{a}) \in \mathcal{B}^+_\Hermitian$.
  
  Now given $b\in \mathcal{B}$, then there is some $\tilde{a} \in \mathcal{A}$ with $\Phi(\tilde{a}) = b$.
  As $\mathcal{A}$ is of finite type, there are $\tilde{a}_{\mathrm{n}}, \tilde{a}_{\mathrm{d}} \in \genCstar{\{p^{-1}, a_1p^{-1}, \dots, a_N p^{-1}\}}$
  such that $\tilde{a}_{\mathrm{d}}$ is invertible in $\mathcal{A}$ and $\tilde{a} = \tilde{a}_\mathrm{n}\tilde{a}_\mathrm{d}^{-1}$.
  Using that $\seminorm{\infty}{\Phi(a)} \le \seminorm{\infty}{a}$ for all $a\in \mathcal{A}^\bd$,
  one sees that the image of $\genCstar{\{p^{-1}, a_1p^{-1}, \dots, a_N p^{-1}\}}$ under $\Phi$
  is contained in the $C^*$\=/subalgebra $\genCstar{\{\Phi(p)^{-1}, \Phi(a_1) \Phi(p)^{-1}, \dots, \Phi(a_N) \Phi(p)^{-1}\}}$
  of $\mathcal{B}^\bd$. This especially applies to $\Phi(\tilde{a}_\mathrm{n})$ and $\Phi(\tilde{a}_\mathrm{d})$, 
  and $\Phi(\tilde{a}_\mathrm{d})$ is invertible in $\mathcal{B}$ with inverse 
  $\Phi(\tilde{a}_\mathrm{d})^{-1} = \Phi(\tilde{a}_\mathrm{d}^{-1})$,
  and $b = \Phi(\tilde{a}) = \Phi(\tilde{a}_\mathrm{n}) \Phi(\tilde{a}_\mathrm{d})^{-1}$ holds.
\end{proof}
\section{Proper \texorpdfstring{\Sus\=/Algebras}{Su*-Algebras} of Continuous Functions}
\label{sec:propersus}
Let $X$ be a (possibly empty) topological space.
In this section, certain intermediate $^*$\=/subalgebras of \Sus\=/algebras of
$\Stetig(X)$ are examined. 
Recall that a function $p \in \Stetig(X)$ is usually said to be a \neu{proper function} if
the preimage of every compact subset of $\CC$ under $p$ is again compact. If
$p \in \Stetig(X)^+_\Hermitian$, this is equivalent to $p^{-1}\big( [0,n] \big)$
being compact for all $n\in \NN$. In this case let 
$U_n \coloneqq p^{-1}\big( {]{-\infty},n[} \big) = p^{-1}\big( {[0,n[}  \big)$
and $K_n \coloneqq p^{-1}\big( [0,n] \big)$, then $U_n \subseteq K_n \subseteq U_{n+1} \subseteq K_{n+1}$
holds for all $n\in \NN$ and $X = \bigcup_{n\in \NN} U_n = \bigcup_{n\in \NN} K_n$.
As $U_n$ is open and $K_n$ compact for all $n\in \NN$, it follows that $X$ is locally
compact and admits an exhaustion by compact sets, i.e. the sequence $(K_n)_{n\in \NN}$
covers $X$ and is strictly increasing in the sense that $K_n$ is contained in the
interior of $K_{n+1}$ for all $n\in \NN$.
\begin{definition}
  Let $X$ be a locally compact Hausdorff space, then a 
  \neu{proper \Sus\=/algebra of continuous functions on $X$} is an intermediate $^*$\=/subalgebra
  $\mathcal{I}$ of $\Stetig(X)$ with the property that there exists a proper
  function $p\in \mathcal{I}^+_\Hermitian$.
\end{definition}
By Proposition~\ref{proposition:intermediateSus}, such an intermediate $^*$\=/subalgebra
of $\Stetig(X)$ indeed is a \Sus\=/algebra. 
If $X$ is a compact Hausdorff space, then there exists exactly one 
proper \Sus\=/algebra of continuous functions on $X$, namely the commutative $C^*$-algebra $\Stetig(X)$ itself.
Another example is, for every $N\in \NN$, the $^*$\=/algebra of all continuous functions on $\RR^N$ that are bounded by a polynomial function.
More generally, it is easy to check that continuous calculi are defined on proper \Sus\=/algebras of continuous functions:
\begin{proposition}
  Let $(X,\mathcal{I},\Phi)$ be a continuous calculus for some $N$-tuple
  of Hermitian elements of some ordered $^*$\=/algebra, then 
  $p \coloneqq \Unit + \pr_1^2 + \dots + \pr_N^2\in\mathcal{I}^+_\Hermitian$
  is a proper function and $\mathcal{I}$
  is a proper \Sus\=/algebra of continuous functions on $X$.
\end{proposition}
The notions of proper functions and
of proper elements in ordered $^*$\=/algebras, Definition~\ref{definition:properelement}, are closely related:
\begin{proposition} \label{proposition:properness}
  Let $\mathcal{I}$ be a proper \Sus\=/algebra of continuous functions on some locally
  compact Hausdorff space $X$, then an element $q\in \mathcal{I}^+_\Hermitian$ is a proper function
  if and only if it is proper in $\mathcal{I}$.
\end{proposition}
\begin{proof}
  As $\mathcal{I}$ is a proper \Sus\=/algebra of continuous functions on $X$, there
  exists a proper function $p \in \mathcal{I}^+_\Hermitian$.
  
  First assume that some element $q\in \mathcal{I}^+_\Hermitian$ is proper in $\mathcal{I}$.
  Then for every $\lambda \in {]0,\infty[}$ there are $\mu \in {[0,\infty[}$ and $h\in \mathcal{I}^+_\Hermitian$
  such that for all $\epsilon \in {]0,\infty[}$ one finds $k\in \NN_0$ for which the estimate
  $p \le \mu \Unit + (q/\lambda)^k + \epsilon h$ holds. Let $K$ be the closed
  (possibly empty) preimage $K \coloneqq q^{-1}([0,\lambda])$, then 
  $p(x) \le \mu+q^k(x)/\lambda^k + \epsilon h(x) \le \mu + 1 + \epsilon h(x)$ for all $x\in K$ independently of $k$,
  hence $p(x) \le \mu + 1$ for all $x\in K$. So $K$ is a closed subset of
  the preimage $p^{-1}([0,\mu+1])$, which is compact because $p$ is a proper function.
  This shows that $K$ is also compact and one concludes that $q$ is a proper function.

  Conversely, assume that some element $q\in \mathcal{I}^+_\Hermitian$ is a proper function.
  Given $\lambda \in {]0,\infty[}$ and $f \in \mathcal{I}^+_\Hermitian$, then let $K$ be the
  compact (possibly empty) preimage $K \coloneqq q^{-1}([0,\lambda+1])$ and define 
  $\mu$ as the maximum of the set $\set{f(x)}{x\in K} \cup \{0\}$ and $g \coloneqq f^2 \in \mathcal{I}^+_\Hermitian$.
  Then for every $\epsilon \in {]0,\infty[}$ and all $x\in X$ one can check that the estimate $f(x)\le \epsilon^{-1} + \epsilon g(x)$ holds
  (consider the cases $f(x) \le \epsilon^{-1}$ and $f(x) \ge \epsilon^{-1}$ separately)
  and there is a $k\in \NN_0$ for which $\epsilon^{-1} \le (\lambda+1)^k/\lambda^k$.
  Thus $f(x) \le q^k(x)/\lambda^k + \epsilon g(x)$ holds for all $x\in X \backslash K$ and consequently $f \le \mu \Unit + (q/\lambda)^k + \epsilon g$,
  which shows that $q$ is proper in $\mathcal{I}$.
\end{proof}
Because of the above Proposition~\ref{proposition:properness} it will no longer be necessarily
to destinguish between ``proper functions'' and ``proper elements'' in proper \Sus\=/algebras
of continuous functions. However, note that Proposition~\ref{proposition:properness} does
not generalize to arbitrary intermediate $^*$\=/algebras of $\Stetig(X)$ for arbitrary
locally compact Hausdorff spaces $X$: For example, the constant $1$-function on $\RR$
is not a proper function, but is proper in $\Stetig(\RR)^\bd$. Nevertheless,
one could extend all functions in $\Stetig(\RR)^\bd$ to the Stone-\v{C}ech compactification
of $\RR$, on which $\Unit$ is a proper function, so the only problem is the ``wrong'' domain
of definition of functions in $\Stetig(\RR)^\bd$.

In a proper \Sus\=/algebra of continuous functions $\mathcal{I}$ on a locally compact Hausdorff space $X$
one can construct the pointwise square root $\sqrt{f} \in \mathcal{I}^+_\Hermitian$ for all $f\in \mathcal{I}^+_\Hermitian$.
Note that $0 \le \sqrt{f} \le \Unit + f$ holds pointwise for all $f\in \mathcal{I}^+_\Hermitian$
so that indeed $\sqrt{f} \in \mathcal{I}$ by Proposition~\ref{proposition:intermediateSus}.
It is also easy to check that this pointwise square root coincides
with the general one discussed in \cite{schoetz:EquivalenceOrderAlgebraicStructure}.
Similarly, one obtains the pointwise absolute value $\abs{f} = \sqrt{f^*f} \in \mathcal{I}^+_\Hermitian$ for all $f\in \mathcal{I}$
and the pointwise positive and negative parts $f_+ \coloneqq (\abs{f}+f)/2  \in \mathcal{I}^+_\Hermitian$ and $f_- \coloneqq (\abs{f}-f)/2 \in \mathcal{I}^+_\Hermitian$
for all $f\in \mathcal{I}_\Hermitian$, i.e.~$f_+(x) = \max\{f(x),0\}$ and $f_-(x) = \max \{ {- f(x)}, 0\}$ for all $x\in X$.
These constructions will be quite helpful in the following.

The central result for proper \Sus\=/algebras of continuous functions is the following ``Nullstellensatz'',
i.e.~a characterization of the vanishing ideals:
\begin{proposition} \label{proposition:nullstellen}
  Let $X$ be a locally compact Hausdorff space, $\mathcal{I}$ a proper \Sus\=/algebra of
  continuous functions on $X$ and $\mathcal{V}$ a closed ideal of $\mathcal{I}$. Define
  the closed subset $Z_{\mathcal{V}} \coloneqq \set[\big]{x\in X}{\forall_{f\in\mathcal{V}}:f(x) = 0}$
  of $X$, which describes the common zeros of the functions in $\mathcal{V}$,
  then $\mathcal{V} = \set[\big]{f\in \mathcal{I}}{\forall_{z\in Z_{\mathcal{V}}}:f(z) = 0}$.
\end{proposition}
\begin{proof}
  As $f(z)=0$ for all $f\in\mathcal{V}$ and all $z\in Z_{\mathcal{V}}$ by definition of $Z_{\mathcal{V}}$, the inclusion
  ``$\subseteq$'' is clear. 
  
  Conversely, consider the case of a function $f\in \mathcal{I}$ that fulfils
  $f(z)=0$ for all $z\in Z_{\mathcal{V}}$, and define $\hat{g} \coloneqq (p+\Unit)^{-1} (\abs{f}+\Unit)^{-1} f \in \Stetig(X)^\bd$,
  where $p\in\mathcal{I}^+_\Hermitian$ is a proper function. In order to prove that
  $f\in \mathcal{V}$, it is sufficient to construct a sequence
  $(g_n)_{n\in \NN}$ in $\mathcal{V}$ that converges against $\hat{g}$
  because $\mathcal{V}$ is a closed ideal of $\mathcal{I}$ by assumption.
  
  Fix $n\in \NN$. The estimate $\abs{\hat{g} (x)} < 1/(1+n)$ holds for all 
  $x\in X \backslash p^{-1}\big([0,n]\big)$. Because of this, $K_n \coloneqq \set[\big]{x\in X}{\abs{\hat{g}(x)} \ge 1/(1+n)}$
  is a closed subset of the compact preimage $p^{-1}\big([0,n]\big)$, hence is again compact.
  Note that $K_n \cap Z_{\mathcal{V}} = \emptyset$, because $\hat{g}$ is non-zero in all
  points of $K_n$ but zero on $Z_{\mathcal{V}}$. So for all $x\in K_n$ there exists an $h_x \in \mathcal{V}$ 
  with $h_x(x) \neq 0$. The open sets $\set[\big]{\tilde{x}\in X}{h_x(\tilde{x})\neq 0}$ for all $x\in K_n$
  cover $K_n$, so there exists $M\in \NN$ and $x_1,\dots,x_M \in K_n$ such that
  $K_n \subseteq \bigcup_{m=1}^M \set[\big]{\tilde{x}\in X}{h_{x_m}(\tilde{x})\neq 0}$. Thus
  $e_n \coloneqq  \lambda\sum_{m=1}^M h_{x_m}^* h_{x_m} \in \mathcal{V}\cap \mathcal{I}^+_\Hermitian$
  fulfils $e_n(x)>0$ for all $x\in K_n$ and all choices of $\lambda \in {]0,\infty[}$,
  hence even $e_n(x) \ge 1$ for all $x\in K_n$ if $\lambda$ is chosen sufficiently
  large because $K_n$ is compact.
  
  The above construction yields a sequence $(e_n)_{n\in \NN}$ of functions in 
  $\mathcal{V}\cap \mathcal{I}^+_\Hermitian$ with the property that $e_n(x) \ge 1$
  for all those $x\in X$ and $n\in \NN$ for which $\abs{\hat{g}(x)} \ge 1/(1+n)$. By multiplying each
  $e_n$ with a suitable function $d_n \in \Stetig(X)^\bd \subseteq \mathcal{I}$ one can now
  construct a sequence $\NN \ni n \mapsto g_n \coloneqq e_nd_n\in\mathcal{V}$
  that converges against $\hat{g}$: One possible choice is
  \begin{equation*}
    d_n(x) \coloneqq \begin{cases}
                       \hat{g}(x) / e_n(x) & \text{for all }x\in X\text{ with }e_n(x)\ge 1 \\
                       \hat{g}(x) & \text{for all }x\in X\text{ with }e_n(x)\le 1\,,
                     \end{cases}
  \end{equation*}
  then $g_n(x) = \hat{g}(x)$ holds for all $x\in X$ with $\abs{\hat{g}(x)} \ge 1/(1+n)$,
  and at the remaining points one has $\abs{g_n(x)} \le \abs{\hat{g}(x)}<1/(1+n)$, hence
  $\abs{g_n(x)-\hat{g}(x)} \le \abs{g_n(x)}+\abs{\hat{g}(x)} < 2/(1+n)$.
\end{proof}
Every locally compact Hausdorff space $X$ is a Tychonoff space \cite[Chap.~5, Thms.~17-18]{kelley:GeneralTopology},
i.e.~$X$ is Hausdorff and for every closed subset $Z$ of $X$ and every $x\in X\backslash Z$
there exists a continuous function $f\colon X\to {[0,1]}$
such that $f(x) = 1$ and $f(z) = 0$ for all $z\in Z$. This yields:
\begin{corollary} \label{corollary:galoisvanishing}
  Let $X$ be a locally compact Hausdorff space and $\mathcal{I}$ a proper \Sus\=/algebra of
  continuous functions on $X$. Assigning to every closed subset $Z$ of $X$ its
  \neu{vanishing ideal} $\mathcal{V}_Z$, i.e.~the closed $^*$\=/ideal
  $\mathcal{V}_Z \coloneqq \set[\big]{f\in\mathcal{I}}{\forall_{z\in Z}:f(z) = 0}$,
  yields a bijection between the closed subsets of $X$ and the closed $^*$\=/ideals of $\mathcal{I}$.
  Its inverse is the assignment of the closed set $Z_\mathcal{V}$ of common zeros to every closed
  $^*$\=/ideal $\mathcal{V}$ of $\mathcal{I}$ as in the previous Proposition~\ref{proposition:nullstellen}.
\end{corollary}
\begin{proof}
  It is easy to check that the evaluation functionals $\mathcal{I} \ni f \mapsto f(x) \in \CC$
  are continuous unital $^*$\=/homomorphisms for all $x\in X$, so the vanishing ideals as the intersection of
  the kernels of such morphisms are closed $^*$\=/ideals of $\mathcal{I}$.
  Moreover, the mapping $Z \mapsto \mathcal{V}_Z$ is surjective onto the closed $^*$\=/ideals of $\mathcal{I}$ 
  by the previous  Proposition~\ref{proposition:nullstellen}, and it is also injective:
  If $Z$ is a closed subset of $X$, then the set of all 
  common zeros of the functions in $\mathcal{V}_Z$ is again $Z$,
  because for every $x\in X\backslash Z$ there exists an $f\in \mathcal{V}_Z$ with $f(x) = 1$
  as $X$ is a Tychonoff space. This also shows that the inverse of the construction of vanishing ideals
  $Z \mapsto \mathcal{V}_Z$ is the construction of the set of common zeros $\mathcal{V}\mapsto Z_\mathcal{V}$.
\end{proof}
For unital $^*$\=/homomorphisms to an Archimedean ordered $^*$\=/algebra one thus finds:
\begin{corollary} \label{corollary:niceMorphisms}
  Let $X$ be a locally compact Hausdorff space, $\mathcal{I}$ a proper \Sus\=/algebra of
  continuous functions on $X$ and $\Phi \colon \mathcal{I} \to \mathcal{A}$ a unital $^*$\=/homomorphism
  to an Archimedean ordered $^*$\=/algebra $\mathcal{A}$. Then $\Phi$ automatically is
  a positive unital $^*$\=/homomorphism and continuous, and there is a unique closed subset $Z$ of $X$ such that
  $\ker \Phi = \mathcal{V}_Z$ like in the previous Corollary~\ref{corollary:galoisvanishing}.
  Moreover, for every $f\in \mathcal{I}_\Hermitian$ with $\Phi(f) \in \mathcal{A}^+_\Hermitian$
  one has $\Phi(f) = \Phi(f_+)$ with $f_+ \in \mathcal{I}^+_\Hermitian$ the positive part
  of $f$. Similarly, for all $\rho \in {]0,\infty[}$ and every $f \in \mathcal{I}$ fulfilling
  $\seminorm{\infty}{\Phi(f)} \le \rho$
  one has $\Phi(\gamma_\rho \circ f ) = \Phi(f)$ with the continuous function
  $\gamma_\rho \colon \CC\to \CC$ defined as
  \begin{align}
    \gamma_\rho(z) \coloneqq \begin{cases}
                               z & \textup{if }\abs{z} \le \rho \\
                               \rho z / \abs{z} & \textup{if }\abs{z} \ge \rho
                             \end{cases}
  \end{align}
  for all $z\in \CC$.
\end{corollary}
\begin{proof}
  Given $f\in \mathcal{I}^+_\Hermitian$, then $\Phi(f) = \Phi(\sqrt{f})^2 \in \mathcal{A}^+_\Hermitian$,
  so $\Phi$ is a positive unital $^*$\=/homomorphism and therefore is continuous. This shows that
  its kernel is a closed $^*$\=/ideal of $\mathcal{I}$, thus $\ker \Phi = \mathcal{V}_Z$ for a 
  unique closed subset $Z$ of $X$ by the previous Corollary~\ref{corollary:galoisvanishing}.
  Moreover, assume that some function $f\in \mathcal{I}_\Hermitian$ fulfils $\Phi(f) \in \mathcal{A}^+_\Hermitian$,
  then the negative part of $f$
  fulfils $0 \le \Phi\big( ( f_- )^3 \big) = - \Phi( f_- ) \Phi(f) \Phi( f_- ) \le 0$
  because $f_- \ge 0$, $f_-f_+ = 0$ and $\Phi(f) \ge 0$, so 
  $( f_- )^3 \in \mathcal{V}_Z$ and consequently also $f_- \in \mathcal{V}_Z$.
  We conclude that $\Phi(f) = \Phi(f_+)$.
  Similarly, for $\rho \in {]0,\infty[}$ and $f \in \mathcal{I}$ with $\seminorm{\infty}{\Phi(f)} \le \rho$ one has
  $\Phi\big( \rho^2 \Unit_{\mathcal{I}} - f^*f \big) =  \rho^2\Unit_\mathcal{A} - \Phi(f)^*\Phi(f) \ge 0$
  so that $( \rho^2\Unit_{\mathcal{I}} - f^*f )_- \in \mathcal{V}_Z$.
  As $\gamma_\rho \circ f$ and $f$ differ only in points $x\in X$ for which $\abs{f(x)} > \rho$,
  this shows that $f - (\gamma_\rho \circ f) \in \mathcal{V}_Z$, so $\Phi(\gamma_\rho \circ f ) = \Phi(f)$.
\end{proof}
Under the bijection between closed subsets and closed $^*$\=/ideals from Corollary~\ref{corollary:galoisvanishing},
the subsets which consist of only one single point correspond to the closed $^*$\=/ideals with codimension $1$.
Thus one finds that unital $^*$\=/homomorphisms between proper \Sus\=/algebras of continuous
functions are always of particularly simple form:
\begin{proposition} \label{proposition:niceMorphisms}
  Let $X$ and $Y$ be two locally compact Hausdorff spaces and let $\mathcal{I}$ and $\mathcal{J}$
  be two proper \Sus\=/algebras of continuous functions on $X$ and $Y$, respectively. Then for every
  unital $^*$\=/homomorphism $\Phi \colon \mathcal{I} \to \mathcal{J}$ there exists a unique
  continuous map $\phi \colon Y \to X$ such that $\Phi(f) = f \circ \phi$ holds for all $f\in \mathcal{I}$.
\end{proposition}
\begin{proof}
  Let such a unital $^*$\=/homomorphism $\Phi \colon \mathcal{I} \to \mathcal{J}$ be given.
  
  For every $y\in Y$, define the unital $^*$\=/homomorphism $\Phi_y \colon \mathcal{I} \to \CC$,
  $f \mapsto \Phi_y(f) \coloneqq \Phi(f)(y)$. Then, as discussed in the previous Corollary~\ref{corollary:niceMorphisms},
  there exists a unique closed subset $Z_y$ of $X$ such that $\ker \Phi_y = \mathcal{V}_{Z_y}$.
  Moreover, $Z_y$ cannot be empty because then $\ker \Phi_y = \mathcal{V}_{Z_y} = \mathcal{I}$ would contradict
  $\Phi_y(\Unit_\mathcal{I}) = \Phi(\Unit_\mathcal{I})(y) = \Unit_{\mathcal{J}}(y) = 1$. Now given any
  $z \in Z_y$, then on the one hand $\ker \Phi_y = \mathcal{V}_{Z_y} \subseteq \mathcal{V}_{\{z\}}$,
  and on the other hand one finds for all $f\in \mathcal{V}_{\{z\}}$ that 
  $f - \Phi_y(f) \Unit_{\mathcal{I}} \in \ker \Phi_y \subseteq \mathcal{V}_{\{z\}}$
  implies that $f(z) - \Phi_y(f) = 0$ holds, so $\Phi_y(f) = f(z) = 0$ and thus $f \in \ker \Phi_y$.
  It follows that $\mathcal{V}_{\{z\}} = \ker \Phi_y = \mathcal{V}_{Z_y}$
  and therefore $\{z\} = Z_y$ by uniqueness of $Z_y$, so~$Z_y$ consists of exactly one single point of $X$.
  
  It is now possible to define a map $\phi \colon Y \to X$, $y \mapsto \phi(y)$,
  with $\phi(y)$ being the unique point for which $\{\phi(y)\} = Z_y$. Then
  $\Phi(f)(y) = \Phi_y(f) = \Phi_y\big( f-f(\phi(y))\Unit_\mathcal{I} \big) + f(\phi(y)) = f(\phi(y))$
  holds for all $f\in \mathcal{I}$ and all $y\in Y$ because $f-f(\phi(y))\Unit_\mathcal{I} \in \mathcal{V}_{\{\phi(y)\}} = \ker \Phi_y$,
  so~$\Phi(f) = f\circ \phi$ for all $f\in \mathcal{I}$. Moreover, $\phi$ is continuous: Indeed, 
  the bijective correspondence between closed subsets of $X$ and closed $^*$\=/ideals of $\mathcal{I}$
  from Corollary~\ref{corollary:galoisvanishing} allows to describe the preimage $\phi^{-1}(Z)$ of some closed subset $Z$ of $X$ as
  \begin{align*}
    \phi^{-1}(Z)
    =
    \set[\big]{y\in Y}{ \phi(y) \in Z }
    =
    \set[\big]{y\in Y}{ \forall_{f\in \mathcal{V}_Z}: f(\phi(y)) = 0 }
    =
    \set[\big]{y\in Y}{ \forall_{f\in \mathcal{V}_Z}: \Phi(f)(y) = 0 }
    ,
  \end{align*}
  which is a closed subset of $Y$ because $\Phi(f)$ is a continuous function on $Y$ for all $f\in\mathcal{I}$.
  
  Finally, it only remains to check that $\phi$ is uniquely determined: So let $\phi \colon Y \to X$
  be any map for which $\Phi(f) = f\circ \phi$ holds and let $y\in Y$ be given, then for all $f\in \ker \Phi_y = \mathcal{V}_{Z_y}$
  one has $0 = \Phi(f)(y) = f(\phi(y))$, i.e.~$\phi(y)$ is a common zero of all $f\in \mathcal{V}_{Z_y}$ and therefore 
  $\phi(y) \in Z_y$ by Corollary~\ref{corollary:galoisvanishing}. As $Z_y$
  consists of exactly one single element of $X$ this means that $\{ \phi(y) \} = Z_y$.
\end{proof}
Note that a statement analogous to the above Proposition~\ref{proposition:niceMorphisms}
is not true if $\mathcal{I} = \Stetig(X)^\bd$ for non-compact $X$ and if $Y = \{*\}$ is the topological
space with only one single point: In this case $\Stetig(\{*\}) \cong \CC$ and
the set of unital $^*$\=/homomorphisms from the commutative $C^*$\=/algebra $\Stetig(X)^\bd$
to $\CC$ is weak\=/$^*$-compact as a consequence of the Banach-Alaoglu theorem
and especially does not consist of only the evaluation functionals $\Stetig(X)^\bd \ni f \mapsto f(x) \in \CC\cong\Stetig(\{*\})$
at points $x\in X$. This means that there are unital $^*$\=/homomorphisms
$\Stetig(X)^\bd \to \Stetig(\{*\})$ that are not of the form
$f \mapsto f \circ \iota$ with $\iota \colon \{ * \} \to X$, $\iota(*) = x$.
So the assumption that $\mathcal{I}^+_\Hermitian$
contains a proper function is crucial.
\section{Universal Continuous Calculi -- Construction}  \label{sec:cccon}
While proper \Sus\=/algebras of continuous functions automatically have very well-behaved order properties,
see e.g.~Corollary~\ref{corollary:niceMorphisms}, this is not true for general ordered $^*$\=/algebras,
especially not for $\CC[t_1,\dots,t_N]$ with the pointwise order. In order to extend the polynomial calculus to a continuous
calculus it will therefore be necessary to show that similar properties are indeed fulfilled in special cases,
which requires some variant of the Positivstellensatz.

The universal continuous calculus for $N \in \NN$ pairwise commuting Hermitian elements $a_1, \dots, a_N$ of a \Sus\=/algebra $\mathcal{A}$ can be constructed
in four steps: First one constructs a positive unital $^*$\=/homomorphism $\Phi_{\mathrm{rt}}$ from a $^*$\=/algebra $\Stetig_{\mathrm{rt}}(\RR^N)$ of certain rational
functions to $\mathcal{A}$. Its restriction to some bounded rational functions can then be extended continuously to a (necessarily positive) unital $^*$\=/homomorphism
$\Phi_{\SS}$ defined on $\Stetig_\SS(\RR^N)$, the continuous functions on $\RR^N$ that extend continuously to the one-point compactification $\RR^N \cup \{\infty\} \cong \SS^N$.
In a third step, the domain of definition of $\Phi_\SS$ is extended to a proper \Sus\=/algebra $\mathcal{I}_\RR$ of continuous functions on $\RR^N$,
which is essentially the largest possible one that admits an extension $\Phi_\RR \colon \mathcal{I}_\RR \to \mathcal{A}$ of $\Phi_\SS$.
The final step is to divide out the kernel of $\Phi_\RR$ and to interpret the resulting quotient $\mathcal{I}_\RR / \ker \Phi_\RR$ 
as a proper \Sus\=/algebra of continuous functions on a closed subset $\spec(a_1,\dots,a_N)$ of $\RR^N$. It will also become clear from the
construction that the universal continuous calculus for $a_1,\dots,a_N$ maps to the bicommutant $\{a_1,\dots,a_N\}''$.
\begin{lemma} \label{lemma:rationalcalculus}
  If $\pi \in \CC[s,t_1,\dots,t_N]_\Hermitian$ with $N\in \NN$ fulfils
  $\pi\big((1+x_1^2+\dots+x_N^2)^{-1},x_1,\dots,x_N\big) \ge 0$ for all $x \in \RR^N$, then
  $\pi\big((\Unit+a_1^2+\dots+a_N^2)^{-1},a_1,\dots,a_N\big) \in \mathcal{A}^+_\Hermitian$ holds for every
  Archimedean ordered $^*$\=/algebra $\mathcal{A}$ and every pairwise commuting $N$-tuple $a_1,\dots,a_N \in \mathcal{A}_\Hermitian$
  for which $\Unit+a_1^2+\dots+a_N^2$ is invertible.
\end{lemma}
\begin{proof}
  Let $\pi \in \CC[s,t_1,\dots,t_N]_\Hermitian$ be such a polynomial fulfilling
  $\pi\big((1+x_1^2+\dots+x_N^2)^{-1},x_1,\dots,x_N\big) \ge 0$ for all $x \in \RR^N$.
  There are $K,L\in \NN_0$ for which $\pi$ can be expanded as 
  \begin{align*}
    \pi
    =
    \sum_{\substack{k \in \NN_0 \\ k \le K}}
    \sum_{\substack{\ell \in (\NN_0)^N \\ \ell_1 + \dots+ \ell_N \le L}}
    \pi_{k,\ell} \,s^k t_1^{\ell_1} \cdots t_N^{\ell_N}
  \end{align*}
  with coefficients $\pi_{k,\ell} \in \RR$. Using these coefficients, one can now construct a homogeneous polynomial
  $\rho \in \CC[u_1, \dots, u_N, v_1,\dots,v_N,w]_\Hermitian$ of degree $2K+L$ as
  \begin{align*}
    \rho \coloneqq
    \sum_{\substack{k \in \NN_0 \\ k \le K}}
    \sum_{\substack{\ell \in (\NN_0)^N \\ \ell_1 + \dots+ \ell_N \le L}}
    \pi_{k,\ell}\,
    w^{L-\ell_1-\dots-\ell_N+2k}
    \big(w^2+(u_1-v_1)^2 + \dots + (u_N-v_N)^2\big)^{K-k}
    \prod_{n=1}^N (u_n-v_n)^{\ell_n}
    .
  \end{align*}
  Then $z^{-(2K+L)} \rho(x,y,z) = \rho (x/z, y/z, 1) = 
  \lambda^{K} \pi \big(\lambda^{-1},(x_1-y_1)/z_1, \dots, (x_N-y_N)/z_N\big) \ge 0$
  with $\lambda \coloneqq 1 + (x_1-y_1)^2 / z_1^2 + \dots + (x_N-y_N)^2 / z_N^2$
  holds for all $(x,y,z) \in \RR^N \times \RR^N \times (\RR\backslash\{0\})$, which 
  shows that $\rho(x,y,z) \ge 0$ for all $(x,y,z) \in \RR^N \times \RR^N \times {]0,\infty[}$,
  and thus even for all $(x,y,z) \in \RR^N \times \RR^N \times {[0,\infty[}$ by continuity.
  Fix $\epsilon \in {]0,\infty[}$, then the homogeneous polynomial $\rho + \epsilon(u_1+\dots+u_N+v_1+\dots+v_N+w)^{2K+L}$
  is strictly positive on $( {[0,\infty[^N} \times {[0,\infty[^N} \times {[0,\infty[} ) \backslash \{(0,0,0)\}$,
  so by a theorem of P\'olya, \cite[Thm.~56]{hardy.littlewood.polya:inequalities},
  there exists an $M \in \NN_0$ such that 
  \begin{align*}
    \sigma_\epsilon \coloneqq (u_1+\dots+u_N+v_1+\dots+v_N+w)^M \Big( \rho + \epsilon(u_1+\dots+u_N+v_1+\dots+v_N+w)^{2K+L} \Big)
  \end{align*}
  can be expanded as
  \begin{align*}
    \sigma_\epsilon
    =
    \sum_{\substack{e,f \in (\NN_0)^N \\ e_1 + \dots + e_N + f_1 + \dots + f_N \le 2K+L+M}}
    \sigma_{\epsilon;e,f} \,u_1^{e_1} \cdots u_N^{e_N} v_1^{f_1} \cdots v_N^{f_N} w^{2K+L+M-e_1 - \dots - e_N - f_1 - \dots - f_N}
  \end{align*}
  with positive coefficients $\sigma_{\epsilon;e,f} \in {[0,\infty[}$.
  
  Now let $\mathcal{A}$ be any Archimedean ordered $^*$\=/algebra, $a_1,\dots,a_N \in \mathcal{A}_\Hermitian$
  pairwise commuting and such that $p \coloneqq \Unit+a_1^2 + \dots + a_N^2$ is invertible in $\mathcal{A}$.
  Note that $p$ and hence also $p^{-1}$ commute with all $a_1,\dots,a_N$.
  Define the element
  \begin{align*}
    b_\epsilon
    \coloneqq
    \sigma_\epsilon \bigg( 
      \frac{(\chi\Unit+\chi^{-1}a_1)^2}{4}, \dots, \frac{(\chi\Unit+\chi^{-1}a_N)^2}{4},
      \frac{(\chi\Unit-\chi^{-1}a_1)^2}{4}, \dots, \frac{(\chi\Unit-\chi^{-1}a_N)^2}{4},
      \Unit
    \bigg)
  \end{align*}
  of $\mathcal{A}$, where $\chi \in {]0,\infty[}$ will be specified later on. Then $b_\epsilon$
  is a sum of squares of Hermitian elements of $\mathcal{A}$, so $b_\epsilon \in \mathcal{A}_\Hermitian^+$.
  In order to simplify the expression for $b_\epsilon$, note that
  $(\chi\Unit+\chi^{-1}a_n)^2/4 - (\chi\Unit-\chi^{-1}a_n)^2/4 = a_n$, and similarly
  $(\chi\Unit+\chi^{-1}a_n)^2/4 + (\chi\Unit-\chi^{-1}a_n)^2/4 = (\chi^4 \Unit + a_n^2)/(2\chi^2)$
  holds for all $n\in \{1,\dots,N\}$, so
  \begin{align*}
    \Unit + \sum_{n=1}^N \bigg( \frac{(\chi\Unit+\chi^{-1}a_n)^2}{4} + \frac{(\chi\Unit-\chi^{-1}a_n)^2}{4} \bigg)
    =
    \frac{1}{2\chi^2} \big( (N\chi^4 + 2\chi^2) \Unit + a_1^2 + \dots + a_N^2 \big)
    .
  \end{align*}
  One can choose $\chi \in {]0,\infty[}$ in such a way that $N\chi^4 + 2\chi^2 = 1$, namely $\chi = ( (N^{-2} + N^{-1})^{1/2} - N^{-1} )^{1/2}$,
  and with this choice one now finds that
  \begin{align*}
    b_\epsilon
    =
    (2\chi^2)^{-M}p^M \Big( p^K \pi\big( p^{-1}, a_1, \dots,a_N \big)  +\epsilon (2\chi^2)^{-(2K+L)}p^{2K+L} \Big)
    .
  \end{align*}
  As $b_\epsilon$ by construction commutes with all $a_n$, $n\in \{1,\dots,N\}$, so that $b_\epsilon a_n^2 = a_n b_\epsilon a_n \ge 0$,
  it follows that
  \begin{align*}
    p^K \pi\big( p^{-1}, a_1, \dots,a_N \big)  +\epsilon (2\chi^2)^{-(2K+L)}p^{2K+L}
    =
    (2\chi^2)^{M} p^{-M} \underbrace{b_\epsilon ( \Unit+a_1^2+ \dots+a_N^2)^M}_{\ge 0} p^{-M}
    \ge
    0
    .
  \end{align*}
  But as this is true for all $\epsilon \in {]0,\infty[}$ and as $\mathcal{A}$ is Archimedean,
  one has $p^K \pi( p^{-1}, a_1, \dots,a_N ) \ge 0$, and thus, by a similar argument,
  \begin{align*}
    \pi( p^{-1}, a_1, \dots,a_N ) = p^{-K} \underbrace{ p^K \pi( p^{-1}, a_1, \dots,a_N ) (\Unit+a_1^2 + \dots + a_N^2)^K}_{\ge 0} p^{-K}
  \end{align*}
  because $p^K \pi( p^{-1}, a_1, \dots,a_N )$ commutes with all $a_1, \dots,a_N$.
\end{proof}
One can now construct a ``rational calculus'':
\begin{definition}
  For $N\in \NN$ let $\Stetig_{\mathrm{rt}}(\RR^N)$ be the unital $^*$\=/subalgebra of $\Stetig(\RR^N)$
  that is generated by the functions $\pr_1, \dots, \pr_N$ and $(\Unit + \pr_1^2 + \dots + \pr_N^2 )^{-1}$,
  endowed with the order on Hermitian elements that it inherits from $\Stetig(\RR^N)$, i.e.~the pointwise one.
\end{definition}
\begin{proposition} \label{proposition:rationalcalculus}
  Let $\mathcal{A}$ be an Archimedean ordered $^*$\=/algebra, $N\in \NN$ and $a_1,\dots,a_N\in \mathcal{A}_\Hermitian$
  pairwise commuting, and assume that $\Unit + a_1^2 + \dots + a_N^2$ has a multiplicative inverse $(\Unit + a_1^2 + \dots + a_N^2)^{-1}  \in \mathcal{A}$.
  Then there exists a unique positive unital $^*$\=/homomorphism
  $\Phi_{\mathrm{rt}} \colon \Stetig_{\mathrm{rt}}(\RR^N) \to \mathcal{A}$ fulfilling $\Phi_{\mathrm{rt}}(\pr_n) = a_n$
  for all $n\in \{1,\dots,N\}$. Moreover, $\Phi_{\mathrm{rt}}(r) \in \{a_1,\dots,a_N\}''$ for all $r\in \Stetig_\mathrm{rt}(\RR^N)$.
\end{proposition}
\begin{proof}
  Let $[\argument]_{\mathrm{rt}} \colon \CC[s,t_1,\dots,t_N] \to \Stetig_\mathrm{rt}(\RR^N)$, $\pi \mapsto [\pi]_{\mathrm{rt}}$ be the unital $^*$\=/homomorphism
  that maps $s$ to $(\Unit_{\Stetig_\mathrm{rt}(\RR^N)}+\pr_1^2 + \dots + \pr_N^2)^{-1}$ and $t_n$ to $\pr_n$ for all $n\in \{1,\dots,N\}$,
  then $[\argument]_{\mathrm{rt}}$ maps surjectively onto $\Stetig_{\mathrm{rt}}(\RR^N)$
  by definition of $\Stetig_{\mathrm{rt}}(\RR^N)$.

  If
  a polynomial $\pi \in \CC[s,t_1,\dots,t_N]_\Hermitian$ fulfils $[\pi]_{\mathrm{rt}} \in \Stetig_\mathrm{rt}(\RR^N)^+_\Hermitian$,
  then $[\pi]_{\mathrm{rt}}$ is pointwise positive on $\RR^N$ by definition of the order on $\Stetig_\mathrm{rt}(\RR^N)_\Hermitian$,
  so $\pi\big( (1+x_1^2 + \dots + x_N^2)^{-1}, x_1, \dots,x_N \big) = [\pi]_{\mathrm{rt}}(x_1,\dots,x_N) \ge 0$ for all $x\in \RR^N$,
  and Lemma~\ref{lemma:rationalcalculus} shows that $\pi\big( (\Unit_A+a_1^2 + \dots + a_N^2)^{-1}, a_1, \dots,a_N \big) \in \mathcal{A}_\Hermitian^+$.
  Especially if $[\pi]_{\mathrm{rt}} = 0$, then $\pi\big( (\Unit_A+a_1^2 + \dots + a_N^2)^{-1}, a_1, \dots,a_N \big) \in \mathcal{A}_\Hermitian^+ \cap (-\mathcal{A}_\Hermitian^+) = \{0\}$.
  More generally, if $[\pi]_{\mathrm{rt}} = 0$ for any (not necessarily Hermitian) $\pi \in \CC[s,t_1,\dots,t_N]$, then
  it follows from
  $[\RE(\pi)]_{\mathrm{rt}} = 0 = [\IM(\pi)]_{\mathrm{rt}}$ that 
  $\pi\big( (\Unit_A+a_1^2 + \dots + a_N^2)^{-1}, a_1, \dots,a_N \big) = 0$. All of this together shows that the map
  $\Phi_\mathrm{rt} \colon \Stetig_{\mathrm{rt}}(\RR^N) \to \mathcal{A}$,
  \begin{align*}
    [\pi]_{\mathrm{rt}} \mapsto \Phi_\mathrm{rt}\big( [\pi]_{\mathrm{rt}} \big) \coloneqq \pi \big( (\Unit_A+a_1^2 + \dots + a_N^2)^{-1}, a_1, \dots,a_N \big)
  \end{align*}
  is a well-defined positive unital $^*$\=/homomorphism and one has $\Phi_\mathrm{rt}(\pr_n) =  \Phi_\mathrm{rt}\big( [t_n]_{\mathrm{rt}} \big) = a_n$
  for all $n\in \{1,\dots,N\}$ and 
  $\Phi_\mathrm{rt}\big( (\Unit_{\Stetig_\mathrm{rt}(\RR^N)}+\pr_1^2 + \dots + \pr_N^2)^{-1} \big) = \Phi_\mathrm{rt}\big( [s]_{\mathrm{rt}} \big) = (\Unit_A+a_1^2 + \dots + a_N^2)^{-1}$.
  
  As $\Stetig_{\mathrm{rt}}(\RR^N)$ is generated as a unital $^*$\=/algebra by $\pr_1, \dots,\pr_N$ and $(\Unit_{\Stetig_\mathrm{rt}(\RR^N)}+\pr_1^2 + \dots + \pr_N^2)^{-1}$,
  a unital $^*$\=/homomorphism $\Phi_{\mathrm{rt}} \colon \Stetig_\mathrm{rt}(\RR^N) \to \mathcal{A}$
  is uniquely determined by fixing $\Phi_{\mathrm{rt}}(\pr_1), \dots, \Phi_{\mathrm{rt}}(\pr_N)$ because the identity
  $\Phi_{\mathrm{rt}} \big( (\Unit_{\Stetig_\mathrm{rt}(\RR^N)}+\pr_1^2 + \dots + \pr_N^2)^{-1} \big) =
  \big(\Unit_A + \Phi_\mathrm{rt}(\pr_1)^2 + \dots + \Phi_\mathrm{rt}(\pr_N)^2\big){}^{-1}$ necessarily holds. Moreover, as $(\Unit_A+a_1^2 + \dots + a_N^2)^{-1} \in \{a_1,\dots,a_N\}''$ and $a_n \in  \{a_1,\dots,a_N\}''$ for all $n\in \{1,\dots,N\}$
  it also follows that $\Phi_{\mathrm{rt}}(r) \in \{a_1,\dots,a_N\}''$ for all $r\in \Stetig_\mathrm{rt}(\RR^N)$.
\end{proof}
This ``rational calculus'' can now be extended to continuous uniformly bounded functions whose limit at $\infty$ exists:
\begin{definition} \label{definition:sphericalcalculusDomain}
  For $N\in \NN$ let $\SS^N$ be the one-point compactification of $\RR^N$, i.e.~$\SS^N \coloneqq \RR^N \cup \{\infty\}$ as a set
  and a subset $U$ of $\SS^N$ is open if and only if either $U$ is an open subset of $\RR^N$ or if 
  $\SS^N \backslash U$ is a compact subset of $\RR^N$. Moreover, define
  \begin{align}
    \Stetig_\SS(\RR^N) \coloneqq \set[\big]{f\at{\RR^N}}{f \in \Stetig(\SS^N)}
  \end{align}
  where $\argument\at{\RR^N}$ denotes restriction of a function defined on $\SS^N$ to $\RR^N$.
\end{definition}
Some important properties of $\Stetig_\SS(\RR^N)$ are:
\begin{lemma} \label{lemma:sphericalcalculus}
  Given $N\in \NN$, then $\Stetig_\SS(\RR^N)$ is a closed unital $^*$\=/subalgebra of $\Stetig(\RR^N)^\bd$.
  If some function $f\in \Stetig(\RR^N)$ vanishes at $\infty$, i.e.~if it has the property that for every $\epsilon \in {]0,\infty[}$
  there is a compact subset $K$ of $\RR^N$ such that $\abs{f(x)} < \epsilon$ holds for all $x\in \RR^N \backslash K$,
  then $f\in\Stetig_\SS(\RR^N)$; especially $r_0 \coloneqq (\Unit+\pr_1^2 + \dots + \pr_N^2)^{-1} \in \Stetig_\SS(\RR^N)$
  and $r_n \coloneqq \pr_n (\Unit+\pr_1^2 + \dots + \pr_N^2)^{-1} \in \Stetig_\SS(\RR^N)$ for all $n\in\{1,\dots,N\}$.
  Moreover, $\Stetig_\SS(\RR^N)$ is the $C^*$\=/subalgebra of $\Stetig(\RR^N)^\bd$ that is generated by $\{r_0,r_1,\dots,r_N\}$
  and the unital $^*$\=/subalgebra $\Stetig_\SS(\RR^N) \cap \Stetig_\mathrm{rt}(\RR^N)$ of $\Stetig_\SS(\RR^N)$
  is dense in $\Stetig_\SS(\RR^N)$.
\end{lemma}
\begin{proof}
  As the restriction map $\argument\at{\RR^N} \colon \Stetig(\SS^N) \to \Stetig(\RR^N)$ is a positive unital $^*$\=/homomorphism,
  its image $\Stetig_\SS(\RR^N)$ is a unital $^*$\=/subalgebra of $\Stetig(\RR^N)$. As $\Stetig(\SS^N)^\bd = \Stetig(\SS^N)$ due 
  to the compactness of $\SS^N$, it even follows that $\Stetig_\SS(\RR^N) \subseteq \Stetig(\RR^N)^\bd$ by positivity of $\argument\at{\RR^N}$.
  Moreover, this positive unital $^*$\=/homomorphism $\argument\at{\RR^N}$ is injective and even an order embedding because 
  $\RR^N$ is dense in $\SS^N$. This implies that the restriction map is isometric, i.e.~$\seminorm{\infty}{f\at{\RR^N}} = \seminorm{\infty}{f}$
  for all $f\in \Stetig(\SS^N)$, and from completeness of $\Stetig(\SS^N)$ it now follows that $\Stetig_\SS(\RR^N)$
  is closed in $\Stetig(\RR^N)^\bd$.
  
  If some function $f\in \Stetig(\RR^N)^\bd$ has the property that for every $\epsilon \in {]0,\infty[}$
  there is a compact subset $K$ of $\RR^N$ such that $\abs{f(x)} \le \epsilon$ holds for all $x\in \RR^N \backslash K$,
  then the function $f^\mathrm{ext} \colon \SS^N \to \CC$ that is defined as $f^\mathrm{ext}(x) \coloneqq f(x)$ for all
  $x\in \RR^N$ and $f^\mathrm{ext}(\infty) \coloneqq 0$ is continuous, and so $f = f^\mathrm{ext}\at{\RR^N} \in \Stetig_\SS(\RR^N)$.
  It is easy to check that this condition is fulfilled by the functions $r_n$ with $n\in\{0,\dots,N\}$ so that
  one can construct such extensions $r_n^{\mathrm{ext}} \in \Stetig(\SS^N)$ for which $r_n = r_n^{\mathrm{ext}} \at{\RR^N} \in \Stetig_\SS(\RR^N)$.
  
  Finally, note that these extended functions $r_n^{\mathrm{ext}}$ form a point-separating set of functions on $\SS^N$,
  i.e.~for all $x,y\in \SS^N$ with $x\neq y$ there is an $n\in \{0,1,\dots,N\}$ such that $r_n^{\mathrm{ext}}(x) \neq r_n^{\mathrm{ext}}(y)$:
  Indeed, if both $x$ and $y$ are elements of $\RR^N$ but $x\neq y$, then there is an $n\in \{1,\dots,N\}$ such that $\pr_n(x) \neq \pr_n(y)$
  and therefore $r_n^{\mathrm{ext}}(x) \neq r_n^{\mathrm{ext}}(y)$; if exactly one of $x$ and $y$ equals $\infty$, say $x=\infty$ and $y\in \RR^N$,
  then $r_0^{\mathrm{ext}}(x) = 0 \neq r_0(y) = r_0^{\mathrm{ext}}(y)$. Thus by the Stone-Weierstraß theorem, the unital $^*$\=/subalgebra
  of $\Stetig(\SS^N)$ that is generated by these extended functions $\{r_0^{\mathrm{ext}},r_1^{\mathrm{ext}},\dots,r_N^{\mathrm{ext}}\}$ is
  dense in $\Stetig(\SS^N)$. From the continuity of the restriction map $\argument\at{\RR^N}$ it follows
  that the unital $^*$\=/subalgebra of $\Stetig_\SS(\RR^N)$ that is generated by the functions $\{r_0,r_1,\dots,r_N\}$ is dense in $\Stetig_\SS(\RR^N)$,
  so $\genCstar{\{r_0,r_1,\dots,r_N\}} = \Stetig_\SS(\RR^N)$ because $\Stetig_\SS(\RR^N)$ is also closed in $\Stetig(\RR^N)^\bd$.
  Moreover, as $r_n \in \Stetig_\SS(\RR^N) \cap \Stetig_\mathrm{rt}(\RR^N)$ for all $n\in \{0,1,\dots,N\}$, this also means that
  the unital $^*$\=/subalgebra $\Stetig_\SS(\RR^N) \cap \Stetig_\mathrm{rt}(\RR^N)$ of $\Stetig_\SS(\RR^N)$
  is dense in $\Stetig_\SS(\RR^N)$.
\end{proof}
Recall that every positive unital $^*$\=/homomorphism between Archimedean ordered $^*$\=/algebras is automatically continuous, and that \Sus\=/algebras
by definition are complete and contain a multiplicative inverse for every element of the form $\Unit+a_1^2 + \dots+a_N^2$ with Hermitian
elements $a_1,\dots,a_N$ and $N\in \NN$. The above Lemma~\ref{lemma:sphericalcalculus} therefore shows that on \Sus\=/algebras, one can continuously extend 
the rational calculus from Proposition~\ref{proposition:rationalcalculus} to $\Stetig_\SS(\RR^N)$:
\begin{definition} \label{definition:sphericalcalculus}
  Let $\mathcal{A}$ be a \Sus\=/algebra, $N\in \NN$ and $a_1,\dots,a_N\in \mathcal{A}_\Hermitian$ pairwise commuting, then define
  $\Phi_\SS \colon \Stetig_\SS(\RR^N) \to \mathcal{A}$ as the continuous extension of the restriction of the positive unital $^*$\=/homomorphism
  $\Phi_{\mathrm{rt}} \colon \Stetig_{\mathrm{rt}} (\RR^N) \to \mathcal{A}$ from Proposition~\ref{proposition:rationalcalculus}
  to $\Stetig_\SS(\RR^N) \cap \Stetig_\mathrm{rt}(\RR^N)$.
\end{definition}
\begin{proposition} \label{proposition:sphericalcalculus}
  Let $\mathcal{A}$ be a \Sus\=/algebra, $N\in \NN$ and $a_1,\dots,a_N\in \mathcal{A}_\Hermitian$ pairwise commuting, then the map 
  $\Phi_\SS \colon \Stetig_\SS(\RR^N) \to \mathcal{A}$ from the previous Definition~\ref{definition:sphericalcalculus}
  is a positive unital $^*$\=/homomorphism and $\Phi_\SS(f) \in \{a_1,\dots,a_N\}''$ for all $f\in\Stetig_\SS(\RR^N)$.
  The image of $\Stetig_\SS(\RR^N)$ under $\Phi_\SS$ is the $C^*$\=/subalgebra of
  $\mathcal{A}^\bd$ that is generated by $\{ p^{-1}, a_1 p^{-1}, \dots, a_N p^{-1} \}$,
  where $p \coloneqq \Unit_\mathcal{A} + a_1^2 + \dots + a_N^2 \in \mathcal{A}$.
\end{proposition}
\begin{proof}
  $\Phi_\SS$ is defined as the $\seminorm{\infty}{\argument}$-continuous extension of a $\seminorm{\infty}{\argument}$-continuous
  unital $^*$\=/homomorphism, so continuity of multiplication and $^*$\=/involution with respect to the norm $\seminorm{\infty}{\argument}$
  guarantee that $\Phi_\SS$ is again a unital $^*$\=/homomorphism. It is automatically positive: If $f\at{\RR^N}$ is pointwise
  positive for some $f\in \Stetig(\SS^N)$, then $f$ is also pointwise positive because $\RR^N$ is dense in $\SS^N$.
  Because of this, the pointwise square root $\sqrt{f}\in \Stetig(\SS^N)^+_\Hermitian$ is well-defined and $\Phi_\SS(f\at{\RR^N}) = \Phi_\SS(\sqrt{f}\at{\RR^N})^2 \ge 0$.
  
  Moreover, as the image of $\Phi_\mathrm{rt}$ is a subset of the bicommutant $\{a_1,\dots,a_N\}''$ by Proposition~\ref{proposition:rationalcalculus}
  and as $\{a_1,\dots,a_N\}''$ is closed in
  $\mathcal{A}$ by \cite[Prop.~5]{schoetz:EquivalenceOrderAlgebraicStructure}, the image of $\Phi_\SS$ also lies in $\{a_1,\dots,a_N\}''$.

  Finally, the image of $\Stetig_\SS(\RR^N)$ under $\Phi_\SS$ is $\genCstar{ \{ p^{-1}, a_1 p^{-1}, \dots, a_N p^{-1} \} }$:
  On the one hand, $\Stetig_\SS(\RR^N)$ is the $C^*$\=/subalgebra of $\Stetig(\RR^N)^\bd$ that is generated by 
  $r_0 = ( \Unit_{\Stetig(\RR^N)} + \pr_1^2 + \dots + \pr_N^2 )^{-1}$ and all $r_n = \pr_n r_0$ with $n\in \{1,\dots,N\}$
  by Lemma~\ref{lemma:sphericalcalculus}, so its image under the $\seminorm{\infty}{\argument}$-continuous map $\Phi_\SS$ is 
  a subset of $\genCstar{ \{ p^{-1}, a_1 p^{-1}, \dots, a_N p^{-1} \} }$. On the other hand, let
  $b \in \genCstar{ \{ p^{-1}, a_1 p^{-1}, \dots, a_N p^{-1} \} }$ be given. Then there exists a convergent sequence
  $(c_k)_{k\in \NN}$ in the unital $^*$\=/subalgebra of $\mathcal{A}^\bd$ that is generated by $\{ p^{-1}, a_1 p^{-1}, \dots, a_N p^{-1} \}$
  with limit $\lim_{k\to \infty} c_k = b$. One can even arrange that $\seminorm{\infty}{b-c_{k}} \le 2^{-(k+1)}$ for all $k\in \NN$,
  hence $\seminorm{\infty}{c_{k+1}-c_k} \le 2^{-k}$. This allows to construct a Cauchy-sequence $(g_k)_{k\in \NN}$ in $\Stetig_\SS(\RR^N)$
  fulfilling $\Phi_\SS(g_k) = c_k$ for all $k\in \NN$ as follows: Let $g_1 \in \Stetig_\SS(\RR^N)$ be any preimage of $c_1$ under
  $\Phi_\SS$, which exists by construction of the sequence $(c_k)_{k\in \NN}$. If $g_1, \dots, g_K$ have been defined for some $K\in \NN$,
  let $g_{K+1} \coloneqq g_K + ( \gamma_{2^{-K}} \circ \delta_K )\at{\RR^N}$ with $\delta_K \in \Stetig(\SS^N)$ a preimage of 
  $c_{K+1} - c_K$ under $\Phi_\SS \circ \argument\at{\RR^N} \colon \Stetig(\SS^N) \to \mathcal{A}$ and $\gamma_{2^{-K}} \colon \CC\to \CC$
  like in Corollary~\ref{corollary:niceMorphisms}. Then $\Phi_\SS\big( g_K + ( \gamma_{2^{-K}} \circ \delta_K )\at{\RR^N} \big) = c_K + \Phi_\SS(\delta_K \at{\RR^N}) = c_{K+1}$
  by Corollary~\ref{corollary:niceMorphisms} and because $\seminorm{\infty}{\Phi_\SS(\delta_K \at{\RR^N})} = \seminorm{\infty}{c_{K+1}-c_K} \le 2^{-K}$,
  and one also finds that the estimate
  $\seminorm{\infty}{g_{K+1} - g_K} = \seminorm{\infty}{( \gamma_{2^{-K}} \circ \delta_K )\at{\RR^N} } \le 2^{-K}$ holds because $\abs{\gamma_{2^{-K}}(z)} \le 2^{-K}$
  for all $z\in \CC$.
  The resulting sequence $(g_k)_{k\in \NN}$ therefore is indeed a Cauchy sequence and its limit $f \coloneqq \lim_{k\to\infty} g_k \in \Stetig_\SS(\RR^N)$
  fulfils $\Phi_\SS(f) = \lim_{k\to \infty} \Phi_\SS(g_k) = \lim_{k\to \infty} c_k = b$.
\end{proof}
However, $\Phi_\SS$ does not describe a continuous calculus simply because $\pr_n \notin \Stetig_\SS(\RR^N)$
for $n\in \{1,\dots,N\}$. This will be fixed in the next step:
\begin{proposition} \label{proposition:bigcalculus}
  Let $\mathcal{A}$ be a \Sus\=/algebra, $N\in \NN$ and $a_1,\dots,a_N\in \mathcal{A}_\Hermitian$ pairwise commuting and write
  $p \coloneqq \Unit + \pr_1^2 + \dots + \pr_N^2 \in \Stetig(\RR^ N)$. Construct $\Stetig_\SS(\RR^N)$ and $\Phi_\SS \colon \Stetig_\SS(\RR^N) \to \mathcal{A}$
  like in Definitions~\ref{definition:sphericalcalculusDomain} and \ref{definition:sphericalcalculus}.
  Then for every $f\in \Stetig(\RR^N)$ the function $f^*f + p$ has a pointwise inverse and
  $(f^*f+p)^{-1} \in \Stetig_\SS(\RR^N)$ and $f (f^*f +p)^{-1} \in \Stetig_\SS(\RR^N)$ hold. Moreover:
  \begin{enumerate}
    \item \label{item:bigcalculus:I}
      The set
      \begin{align}
        \mathcal{I}_\RR
        \coloneqq
        \set[\big]{
          f \in \Stetig(\RR^N)
        }{
          \Phi_\SS\big( (f^*f+p)^{-1} \big) \textup{ is invertible in }\mathcal{A} 
        }
      \end{align}
      has the following property: If for some $f\in \Stetig(\RR^N)$ there exists an $h \in \Stetig_\SS(\RR^N)^+_\Hermitian$ such that
      $\Phi_\SS(h)$ has a multiplicative inverse $\Phi_\SS(h)^{-1} \in \mathcal{A}$ and such that $(f^*f+p)^{-1} \ge h$ holds,
      then $f\in\mathcal{I}_\RR$. This especially shows that $\mathcal{I}_\RR$ is an intermediate $^*$\=/subalgebra of 
      $\Stetig(\RR^N)$ and $\Stetig_{\mathrm{rt}}(\RR^N) \subseteq \mathcal{I}_\RR$. Moreover, $\mathcal{I}_\RR$
      is a commutative \Sus\=/algebra of finite type with generators $\pr_1,\dots,\pr_N$.
    \item \label{item:bigcalculus:Phi}
      The map $\Phi_\RR \colon \mathcal{I}_\RR \to \mathcal{A}$,
      \begin{align}
        f \mapsto \Phi_\RR(f) \coloneqq \Phi_\SS\big( f (f^*f +p)^{-1} \big) \Phi_\SS\big( (f^*f +p)^{-1} \big)^{-1}
      \end{align}
      has the following property: If for some $f\in \mathcal{I}_\RR$ there exists an $h \in \Stetig_\SS(\RR^N)$ such that
      $\Phi_\SS(h)$ has a multiplicative inverse $\Phi_\SS(h)^{-1} \in \mathcal{A}$ and such that $fh \in \Stetig_\SS(\RR^N)$ is fulfilled,
      then the identity $\Phi_\RR(f) =  \Phi_\SS(fh)\Phi_\SS(h)^{-1}$ holds. This especially shows that $\Phi_\RR$ is a unital 
      $^*$\=/homomorphism that extends $\Phi_\SS$ and $\Phi_{\mathrm{rt}}$, i.e.~$\Phi_\RR(f) = \Phi_\SS(f)$
      for all $f\in \Stetig_\SS(\RR^N)$ and $\Phi_\RR(r) = \Phi_\mathrm{rt}(r)$ for all $r\in \Stetig_\mathrm{rt}(\RR^N)$.
      Moreover, one has $\Phi_\RR(f) \in \{a_1,\dots,a_N\}''$  for all $f\in\mathcal{I}_\RR$.
    \item \label{item:bigcalculus:calculus}
      The triple $\big(\RR^N, \mathcal{I}_\RR, \Phi_\RR\big)$ is a continuous calculus for $a_1,\dots,a_N$.
  \end{enumerate}
\end{proposition}
\begin{proof}
  Given $f\in \Stetig(\RR^N)$, then $f^*f+p \ge \Unit$, so $f^*f+p$ is invertible in $\Stetig(\RR^N)$.
  For every $\epsilon \in {]0,\infty[}$, the preimage $K_\epsilon \coloneqq p^{-1}([0,\epsilon^{-1}])$
  is a compact subset of $\RR^N$ (i.e.~$p$ is proper) and one finds that 
  $(f^*f+p)^{-1} \at{\RR^N \backslash K_\epsilon} < \epsilon \Unit$ and
  $\abs{ f } (f^*f+p)^{-1} \at{\RR^N \backslash K_{\epsilon}} < \sqrt{\epsilon} \Unit$
  (for the second estimate, consider the two cases $\abs{f(x)} \ge \epsilon^{-1/2}$ and $0 \le \abs{f(x)} \le \epsilon^{-1/2}$ separately).
  So $(f^*f+p)^{-1}, f (f^*f +p)^{-1} \in \Stetig_\SS(\RR^N)$ by Lemma~\ref{lemma:sphericalcalculus}.
  Recall also that $\Phi_\SS \colon \Stetig_\SS(\RR^N) \to \mathcal{A}$ is a positive unital $^*$\=/homomorphism
  by the previous Proposition~\ref{proposition:sphericalcalculus}.
  
  For part \refitem{item:bigcalculus:I} consider $f\in \Stetig(\RR^N)$ and $h\in \Stetig_\SS(\RR^N)^+_\Hermitian$ with the properties that $\Phi_\SS(h)$
  has a multiplicative inverse $\Phi_\SS(h)^{-1} \in \mathcal{A}$ and $(f^*f+p)^{-1} \ge h$. Then 
  $\Phi_\SS\big((f^*f+p)^{-1} h\big) \ge \Phi_\SS(h)^2$ because $(f^*f+p)^{-1} h \ge h^2$,
  so $\Phi_\SS(h)^{-1}\Phi_\SS\big((f^*f+p)^{-1} h\big) \Phi_\SS(h)^{-1}$ is coercive
  and consequently is invertible in $\mathcal{A}$ because $\mathcal{A}$ is symmetric by assumption.
  It follows that $\Phi_\SS\big((f^*f+p)^{-1} h\big)$ also is invertible in $\mathcal{A}$, thus also
  $\Phi_\SS\big((f^*f+p)^{-1}\big)$ with inverse
  $\Phi_\SS\big((f^*f+p)^{-1}\big){}^{-1} = \Phi_\SS(h)\Phi_\SS\big((f^*f+p)^{-1} h\big){}^{-1} = \Phi_\SS\big((f^*f+p)^{-1} h\big){}^{-1} \Phi_\SS(h)$.
  This shows that $f\in \mathcal{I}_\RR$.
  
  For example, as $\Phi_\SS( p^{-1}) = (\Unit+a_1^2+\dots+a_N^2)^{-1}$ by construction of $\Phi_\SS$,
  which is invertible in $\mathcal{A}$, one can choose $h \coloneqq \epsilon p^{-1}$ with $\epsilon \in {]0,\infty[}$.
  So $\Stetig(\RR^N)^\bd \subseteq \mathcal{I}_\RR$
  because $(f^*f+p)^{-1} \ge (1+\seminorm{\infty}{f}^2)^{-1} p^{-1}$ holds for all $f\in \Stetig(\RR^N)^\bd$,
  and also $\pr_n \in \mathcal{I}_\RR$ for all $n\in \{1,\dots,N\}$ because $(\pr_n^2 + p)^{-1} \ge \frac{1}{2} p^{-1}$.
  Similarly, given $f,g\in \mathcal{I}_\RR$, then one can choose $h \coloneqq \frac{1}{2} (f^*f+p)^{-1}(g^*g+p)^{-1} \in \Stetig_\SS(\RR^N)^+_\Hermitian$,
  for which $\Phi_\SS(h) = \frac{1}{2}\Phi_\SS\big((f^*f+p)^{-1}\big) \Phi_\SS\big((g^*g+p)^{-1}\big)$ indeed is invertible in $\mathcal{A}$. The estimates
  \begin{align*}
    (f+g)^*(f+g) + p &\le 2(f^*f + g^*g) + p \le 2 (f^*f + p) (g^*g + p)
  \shortintertext{and}
    (fg)^*(fg) + p &= f^*f g^*g + p \le (f^*f + p) (g^*g + p) 
  \end{align*}
  then show that $\big((f+g)^*(f+g) + p\big){}^{-1} \ge h$ and $\big((fg)^*(fg) + p\big){}^{-1} \ge h$, so $f+g,fg \in \mathcal{I}_\RR$.
  As $\mathcal{I}_\RR$ clearly is stable under the $^*$\=/involution of $\Stetig(\RR^N)$, i.e.~pointwise complex conjugation,
  we conclude that $\mathcal{I}_\RR$ is an intermediate $^*$\=/subalgebra of 
  $\Stetig(\RR^N)$ fulfilling $\pr_n \in \mathcal{I}_\RR$ for all $n\in \{1,\dots,N\}$ and $p^{-1} \in \Stetig(\RR^N)^\bd \subseteq \mathcal{I}_\RR$,
  thus $\Stetig_{\mathrm{rt}}(\RR^N) \subseteq \mathcal{I}_\RR$ because $\Stetig_{\mathrm{rt}}(\RR^N)$ is
  generated by $\pr_1,\dots,\pr_N$ and $p^{-1}$.
  
  Moreover, $\mathcal{I}_\RR$ is a proper \Sus\=/algebra of continuous functions on $\RR^N$ because $p$ is a proper function.
  Proposition~\ref{proposition:properness} thus shows that $p$ is also a proper element of $\mathcal{I}_\RR$ in the sense of 
  Definition~\ref{definition:properelement}. One even finds that $\mathcal{I}_\RR$ is a commutative \Sus\=/algebra of finite type with generators $\pr_1,\dots,\pr_N$:
  Indeed, given $f\in \mathcal{I}_\RR$, then by Lemma~\ref{lemma:sphericalcalculus}, the two functions $f_\mathrm{n} \coloneqq f (f^*f+p)^{-1} \in \Stetig_\SS(\RR^N)$ and 
  $f_\mathrm{d} \coloneqq (f^*f+p)^{-1} \in \Stetig_\SS(\RR^N)$ are elements of the $C^*$\=/subalgebra of $(\mathcal{I}_\RR)^\bd = \Stetig(\RR^N)^\bd$
  that is generated by $\{p^{-1}, \pr_1\,p^{-1}, \dots, \pr_N \,p^{-1}\}$.
  It is also clear that $f_\mathrm{d}$ is invertible in $\mathcal{I}_\RR$ and $f = f_\mathrm{n} f_\mathrm{d}^{-1}$.
  
  For part \refitem{item:bigcalculus:Phi} consider the map $\Phi_\RR \colon \mathcal{I}_\RR \to \mathcal{A}$. 
  Given $f\in \Stetig_\SS(\RR^N)$, then one has
  \begin{align*}
    \Phi_\RR(f)
    =
    \Phi_\SS\big( f (f^*f +p)^{-1} \big) \Phi_\SS\big( (f^*f +p)^{-1} \big)^{-1}
    =
    \Phi_\SS(f) \Phi_\SS\big( (f^*f +p)^{-1} \big) \Phi_\SS\big( (f^*f +p)^{-1} \big)^{-1}
    =
    \Phi_\SS(f)
  \end{align*}
  and so $\Phi_\RR$ extends $\Phi_\SS$. Given functions $f\in \mathcal{I}_\RR$ and $h \in \Stetig_\SS(\RR^N)$
  such that $\Phi_\SS(h)$ has a multiplicative inverse $\Phi_\SS(h)^{-1} \in \mathcal{A}$ and such that $fh \in \Stetig_\SS(\RR^N)$ holds,
  then
  \begin{align*}
    \Phi_\RR(f)
    &=
    \Phi_\SS\big( f (f^*f +p)^{-1} \big) \Phi_\SS\big( (f^*f +p)^{-1} \big)^{-1}
    \\
    &=
    \Phi_\SS(h)^{-1} \Phi_\SS\big( hf(f^*f +p)^{-1} \big) \Phi_\SS\big( (f^*f +p)^{-1} \big)^{-1}
    \\
    &=
    \Phi_\SS(h)^{-1} \Phi_\SS( f h) \Phi_\SS\big( (f^*f +p)^{-1} \big) \Phi_\SS\big( (f^*f +p)^{-1} \big)^{-1}
    \\
    &=
    \Phi_\SS( f h) \Phi_\SS(h)^{-1}
    .
  \end{align*}
  For example, $p^{-1} \in\Stetig_\SS(\RR^N)$ and $\pr_n\,p^{-1} \in \Stetig_\SS(\RR^N)$ for all $n\in \{1,\dots,N\}$
  by Lemma~\ref{lemma:sphericalcalculus}, and $\Phi_\SS(p^{-1}) = \Phi_{\mathrm{rt}}(p^{-1}) = (\Unit+a_1^2 + \dots + a_N^2)^{-1}$
  is invertible in $\mathcal{A}$. So
  \begin{align*}
    \Phi_\RR(\pr_n)
    = 
    \Phi_\SS(\pr_n \,p^{-1}) \Phi_\SS(p^{-1})^{-1}
    =
    \Phi_{\mathrm{rt}}(\pr_n \,p^{-1}) \Phi_{\mathrm{rt}}(p^{-1})^{-1}
    =
    \Phi_{\mathrm{rt}}(\pr_n)
    =
    a_n
    .
  \end{align*}
  Moreover, given $f,g\in \mathcal{I}_\RR$, then $f (f^*f+p)^{-1}, (f^*f+p)^{-1}, g (g^*g+p)^{-1}, (g^*g+p)^{-1} \in \Stetig_\SS(\RR^N)$
  has already been shown, thus $fh,gh,(f+g)h, fgh \in \Stetig_\SS(\RR^N)$ with $h \coloneqq (f^*f + p)^{-1}(g^*g + p)^{-1} \in \Stetig_\SS(\RR^N)$.
  As $\Phi_\SS(h) = \Phi_\SS\big( (f^*f + p)^{-1}\big) \Phi_\SS\big((g^*g + p)^{-1} \big)$ is invertible in $\mathcal{A}$, one finds that
  \begin{align*}
    \Phi_\RR(f+g)
    &=
    \Phi_\SS\big( (f+g) h \big) \Phi_\SS(h)^{-1}
    = 
    \Phi_\SS(fh) \Phi_\SS(h)^{-1} + \Phi_\SS(gh) \Phi_\SS(h)^{-1}
    =
    \Phi_\RR(f) + \Phi_\RR(g)
  \end{align*}
  and similarly also
  \begin{align*}
    \Phi_\RR(fg)
    &=
    \Phi_\SS( fg h ) \Phi_\SS(h)^{-1}
    \\
    &=
    \Phi_\SS\big( f(f^*f+p)^{-1} \big)\Phi_\SS\big( g(g^*g+p)^{-1} \big) \Phi_\SS\big((g^*g+p)^{-1}\big)^{-1} \Phi_\SS\big((f^*f+p)^{-1}\big)^{-1}
    \\
    &=
    \Phi_\RR(f) \Phi_\RR(g)
  \end{align*}
  hold. Compatibility of $\Phi_\RR$ with the $^*$\=/involution follows immediately from the compatibility of $\Phi_\SS$ with the $^*$\=/involution,
  so it has been shown that $\Phi_\RR$ is a unital $^*$\=/homomorphism that extends $\Phi_\SS$, especially $\Phi_\RR(p^{-1}) = \Phi_\SS(p^{-1}) = \Phi_\mathrm{rt}(p^{-1})$,
  and which fulfils $\Phi_\RR(\pr_n) = \Phi_\mathrm{rt}(\pr_n)$ for all $n\in \{1,\dots,N\}$. As $\Stetig_{\mathrm{rt}}(\RR^N)$ is generated
  by $p^{-1}$ and $\pr_1,\dots,\pr_N$, it follows that $\Phi_\RR$ also extends $\Phi_{\mathrm{rt}}$.
  Note also that $\Phi_\RR(f) \in \{a_1,\dots,a_N\}''$ for all $f\in \mathcal{I}_\RR$ because $\Phi_\SS\big(f(f^*f+p)^{-1}\big),\Phi_\SS\big((f^*f+p)^{-1}\big) \in \{a_1,\dots,a_N\}''$
  by Proposition~\ref{proposition:sphericalcalculus}.
  
  Finally for part \refitem{item:bigcalculus:calculus} it only remains to combine the previous results: By part \refitem{item:bigcalculus:I},
  $\mathcal{I}_\RR$ is an intermediate $^*$\=/subalgebra of $\Stetig(\RR^N)$ fulfilling $\pr_n \in \mathcal{I}_\RR$ for all $n\in \{1,\dots,N\}$.
  By part \refitem{item:bigcalculus:Phi}, $\Phi_\RR \colon \mathcal{I}_\RR \to \mathcal{A}$ is a unital $^*$\=/homomorphism and $\Phi_\RR(\pr_n) = \Phi_\mathrm{rt}(\pr_n) = a_n$
  for all $n\in \{1,\dots,N\}$. So $\big(\RR^N, \mathcal{I}_\RR,\Phi_\RR\big)$ is a continuous calculus for $a_1,\dots,a_N$.
\end{proof}
The continuous calculus that has just been constructed is maximal in the following sense:
\begin{corollary} \label{corollary:bigcalculus}
  Let $\mathcal{A}$ be a \Sus\=/algebra, $N\in \NN$ and $a_1,\dots,a_N\in \mathcal{A}_\Hermitian$ pairwise commuting, and 
  let $\big(\RR^N, \mathcal{I}_\RR, \Phi_\RR\big)$ be the continuous calculus for $a_1,\dots,a_N$ that was constructed
  in the previous Proposition~\ref{proposition:bigcalculus}. Moreover, let $(\RR^N, \mathcal{J}, \Psi)$ be
  another continuous calculus for $a_1,\dots,a_N$ that is also defined on a space $\mathcal{J}$ of continuous functions
  on $\RR^N$. Then $\mathcal{J} \subseteq \mathcal{I}_\RR$ and $\Psi(f) = \Phi_\RR(f)$ for all $f\in \mathcal{J}$.
\end{corollary}
\begin{proof}
  As $\Psi(\pr_n) = a_n = \Phi_\RR(\pr_n)$ for all $n\in \NN$, and as $\Stetig_\mathrm{rt}(\RR^N)$
  is by definition generated by $p^{-1},\pr_1,\dots,\pr_N \in \mathcal{I}_\RR \cap \mathcal{J}$
  with $p \coloneqq \Unit + \pr_1^2 + \dots + \pr_N^2 \in \mathcal{I}_\RR \cap \mathcal{J}$,
  it is clear that $\Stetig_\mathrm{rt}(\RR^N) \subseteq \mathcal{I}_\RR \cap \mathcal{J}$ and that $\Psi$ and $\Phi_\RR$ coincide
  on $\Stetig_\mathrm{rt}(\RR^N)$. As $\mathcal{I}_\RR$ and $\mathcal{J}$ are proper \Sus\=/algebras of continuous functions on $\RR^N$,
  the unital $^*$\=/homomorphisms $\Phi_\RR \colon \mathcal{I}_\RR \to \mathcal{A}$ and $\Psi \colon \mathcal{J} \to \mathcal{A}$ 
  are automatically positive by Corollary~\ref{corollary:niceMorphisms}, therefore are continuous and thus coincide even on the closure
  of $\Stetig_\mathrm{rt}(\RR^N)$ in $\mathcal{I}_\RR \cap \mathcal{J}$, which, by Lemma~\ref{lemma:sphericalcalculus},
  contains $\Stetig_\SS(\RR^N)$.
  Now let any $f\in \mathcal{J}$ be given, then $f(f^*f+p)^{-1} \in \Stetig_\SS(\RR^N)$ and $(f^*f+p)^{-1} \in \Stetig_\SS(\RR^N)$
  by the previous Proposition~\ref{proposition:bigcalculus}, and
  $\Phi_\RR\big((f^*f+p)^{-1}\big) = \Psi\big((f^*f+p)^{-1}\big) = \Psi(f^*f+p)^{-1}$ is invertible in $\mathcal{A}$, so $f\in \mathcal{I}_\RR$
  because $\Phi_\RR$ extends $\Phi_\SS$ by the previous Proposition~\ref{proposition:bigcalculus}.
  Similarly,
  $\Psi(f) = \Psi\big(f(f^*f+p)^{-1}\big) \Psi\big((f^*f+p)^{-1}\big){}^{-1} = \Phi_\RR\big(f(f^*f+p)^{-1}\big) \Phi_\RR\big((f^*f+p)^{-1}\big){}^{-1} = \Phi_\RR(f)$
  holds.
\end{proof}
The continuous calculus from Proposition~\ref{proposition:bigcalculus} now yields the universal
continuous calculus by taking a suitable quotient:
\begin{lemma} \label{lemma:largestpossible}
  Let $\mathcal{A}$ be a \Sus\=/algebra, $N\in \NN$ and let $\big(\RR^N, \mathcal{I}_\RR,\Phi_\RR\big)$ be the continuous
  calculus for some pairwise commuting elements $a_1,\dots,a_N\in \mathcal{A}_\Hermitian$ from
  Proposition~\ref{proposition:bigcalculus}. If $f\in \Stetig(\RR^N)^+_\Hermitian$ is coercive 
  and $\Phi_\RR(f^{-1})$ invertible in $\mathcal{A}$, then $f\in \mathcal{I}_\RR$.
\end{lemma}
\begin{proof}
  Given a coercive $f\in \Stetig(\RR^N)^+_\Hermitian$, then there exists $\epsilon \in {]0,1]}$ such that $f \ge \epsilon \Unit$.
  Write again $p \coloneqq \Unit + \pr_{1}^2 + \dots + \pr_{N}^2 \in \Stetig_\mathrm{rt}(\RR^N)^+_\Hermitian \subseteq (\mathcal{I}_\RR)^+_\Hermitian$
  like in  Proposition~\ref{proposition:bigcalculus}, then the pointwise estimates $p \le \epsilon^{-2} f^2 p$ and $f^2 \le \epsilon^{-2} f^2 p$ hold,
  hence $f^2+p \le 2\epsilon^{-2} f^2 p$. Moreover, let $h \coloneqq (2\epsilon^{-2}f^2p)^{-1}$,
  then $0 \le h \le p^{-1}$ holds and therefore $h \in \Stetig_\SS(\RR^N)^+_\Hermitian$ by the criterium of Lemma~\ref{lemma:sphericalcalculus}.
  If $\Phi_\RR(f^{-1})$ is invertible in $\mathcal{A}$, then $\Phi_\RR(h) = \frac{1}{2} \epsilon^2\Phi_\RR(f^{-1})^2 \Phi_\RR(p)^{-1}$
  also is invertible in $\mathcal{A}$, so $f\in \mathcal{I}_\RR$ by part \refitem{item:bigcalculus:I}
  of Proposition~\ref{proposition:bigcalculus} and because $(f^2+p)^{-1}\ge h$.
\end{proof}
\begin{theorem} \label{theorem:univcc}
  Let $\mathcal{A}$ be a \Sus\=/algebra, $N\in \NN$ and $a_1,\dots,a_N\in \mathcal{A}_\Hermitian$ pairwise commuting. Then the universal
  continuous calculus $\big(\spec(a_1,\dots,a_N), \mathcal{F}(a_1,\dots,a_N), \Gamma_{a_1,\dots,a_N}\big)$
  for $a_1,\dots,a_N$ exists and can be constructed as follows:
  \begin{enumerate}
    \item The spectrum of $a_1,\dots,a_N$ is given by
      \begin{align}
        \spec(a_1,\dots,a_N)
        = 
        \RR^N \,\big\backslash \,\set[\Big]{x\in \RR^N}{ \sum\nolimits_{n=1}^N (x_n\Unit_\mathcal{A}-a_n)^2\text{ is a coercive element of }\mathcal{A}^+_\Hermitian }.
        \label{eq:ucc:spec}
      \end{align}
    \item The intermediate $^*$\=/subalgebra $\mathcal{F}(a_1,\dots,a_N)$ of $\Stetig\big(\spec(a_1,\dots,a_N)\big)$
      is
      \begin{align}
        \mathcal{F}(a_1,\dots,a_N) = \set[\big]{ f\at{\spec(a_1,\dots,a_N)} }{ f \in \mathcal{I}_\RR }
        \label{eq:ucc:F}
      \end{align}
      with $\mathcal{I}_\RR$ the intermediate $^*$\=/subalgebra of $\Stetig(\RR^N)$ from 
      Proposition~\ref{proposition:bigcalculus}, part~\refitem{item:bigcalculus:I}.
    \item The unital $^*$\=/homomorphism $\Gamma_{a_1,\dots,a_N} \colon \mathcal{F}(a_1,\dots,a_N) \to \mathcal{A}$
      is determined by
      \begin{align}
        \Gamma_{a_1,\dots,a_N}\big( f\at{\spec(a_1,\dots,a_N)} \big) = \Phi_\RR(f)
        \label{eq:ucc:Gamma}
      \end{align}
      for all $f \in \mathcal{I}_\RR$, with $\Phi_\RR$ the unital $^*$\=/homomorphism
      $\Phi_\RR \colon \mathcal{I}_\RR \to \mathcal{A}$ from 
      Proposition~\ref{proposition:bigcalculus}, part~\refitem{item:bigcalculus:Phi}.
  \end{enumerate}
  This construction has some additional properties: The map $\Gamma_{a_1,\dots,a_N} \colon \mathcal{F}(a_1,\dots,a_N) \to \mathcal{A}$ is
  an injective positive unital $^*$\=/homomorphism and even is an order embedding, and its image is a subset of the bicommutant
  $\{a_1,\dots,a_N\}''$. The space of functions $\mathcal{F}(a_1,\dots,a_N)$ is a commutative \Sus\=/algebra of finite type
  with generators $\pr_1,\dots,\pr_N$, and whenever $f \in \Stetig\big( \spec(a_1,\dots,a_N) \big){}^+_\Hermitian$ is coercive
  and $\Gamma_{a_1,\dots,a_N}(f^{-1})$ invertible in $\mathcal{A}$, then $f\in \mathcal{F}(a_1,\dots,a_N)$.
\end{theorem}
\begin{proof}
  Let $\big(\RR^N, \mathcal{I}_\RR, \Phi_\RR\big)$ be the continuous calculus for $a_1,\dots,a_N$ that was constructed in
  Proposition~\ref{proposition:bigcalculus} and write again 
  $p \coloneqq \Unit + \pr_{1}^2 + \dots + \pr_{N}^2 \in \Stetig_\mathrm{rt}(\RR^N)^+_\Hermitian \subseteq (\mathcal{I}_\RR)^+_\Hermitian$. 
  As $p$ is proper, $\mathcal{I}_\RR$ is a proper \Sus\=/algebra of continuous functions on $\RR^N$, so the results
  from Section~\ref{sec:propersus} apply: By Corollary~\ref{corollary:niceMorphisms},
  the kernel of $\Phi_\RR$ is a closed $^*$\=/ideal of $\mathcal{I}_\RR$ and can be described as
  $\ker \Phi_\RR = \mathcal{V}_Z = \set[\big]{f\in\mathcal{I}_\RR}{f\at{Z} = 0}$, where $Z$ is the closed subset
  $Z = \set[\big]{x\in \RR^N}{f(x) = 0\textup{ for all }f\in\ker \Phi_\RR}$ of $\RR^N$
  by Proposition~\ref{proposition:nullstellen}.
  
  In order to show that the universal continuous calculus for $a_1,\dots,a_N$ exists and that it can be constructed
  as described above, the first step is to show that $Z$ coincides with the right-hand side of \eqref{eq:ucc:spec}, i.e.~that
  \begin{align}
    Z
    = 
    \RR^N \,\big\backslash \,\set[\Big]{x\in \RR^N}{ \sum\nolimits_{n=1}^N (x_n\Unit_\mathcal{A}-a_n)^2\text{ is a coercive element of }\mathcal{A}^+_\Hermitian }
    \label{eq:ucc:spec:internal}   
    \tag{\textup{$*$}}
  \end{align}
  holds.
  From $\Unit \in \mathcal{I}_\RR$ and $\pr_n \in \mathcal{I}_\RR$ for all $n\in \{1,\dots,N\}$ it 
  follows that the continuous function $f_{x,\epsilon} \coloneqq \sum_{n=1}^N (x_n\Unit-\pr_n)^2 - \epsilon \Unit$
  on $\RR^N$ is, for all $x\in \RR^N$ and all $\epsilon \in {]0,\infty[}$, an element of $(\mathcal{I}_\RR)_\Hermitian$.
  Now given $x\in \RR^N$ and $\epsilon \in {]0,\infty[}$ such that $\sum_{n=1}^N (x_n\Unit_\mathcal{A}-a_n)^2 \ge \epsilon \Unit_\mathcal{A}$, 
  then this means that $\Phi_\RR(f_{x,\epsilon}) \in \mathcal{A}^+_\Hermitian$
  and Corollary~\ref{corollary:niceMorphisms} shows that $f_{x,\epsilon;-} \in \ker \Phi_\RR$,
  where $f_{x,\epsilon;-}$ is the negative part of $f_{x,\epsilon}$.
  In this case one has $x \notin Z$ because $f_{x,\epsilon;-}(x) = \epsilon \neq 0$.
  Conversely, if $x \in \RR^N \backslash Z$, then there exists $\epsilon \in {]0,\infty[}$ such that $f_{x,\epsilon}(z) \ge 0$
  for all $z \in Z$ because $\RR^N \backslash Z$ is an open neighbourhood of $x$. So $f_{x,\epsilon}$ and its positive part $f_{x,\epsilon;+}$
  coincide on $Z$, and therefore
  $\sum_{n=1}^N (x_n\Unit_\mathcal{A}-a_n)^2 - \epsilon \Unit_\mathcal{A}  = \Phi(f_{x,\epsilon}) = \Phi(f_{x,\epsilon;+}) \in \mathcal{A}^+_\Hermitian$,
  which shows that $\sum_{n=1}^N (x_n\Unit_\mathcal{A}-a_n)^2$ is coercive. We conclude that identity \eqref{eq:ucc:spec:internal} indeed holds.
  
  The next step is to show that $\mathcal{I}_Z \coloneqq \set[\big]{f\at{Z}}{f\in\mathcal{I}_\RR}$ like in \eqref{eq:ucc:F}
  is an intermediate $^*$\=/subalgebra of $\Stetig(Z)$
  and that $\pr_{Z;n} \in \mathcal{I}_Z$ for all $n \in \{1,\dots,N\}$:
  As the restriction map $\argument \at{Z} \colon \Stetig(\RR^N) \to \Stetig(Z)$ is a unital $^*$\=/homomorphism, it is clear
  that $\mathcal{I}_Z$ is a unital $^*$\=/subalgebra of $\Stetig(Z)$, and $\pr_{Z;n} = \pr_{\RR^N;n} \at{Z} \in \mathcal{I}_Z$
  holds for all $n\in \{1,\dots,N\}$. In order to check that $\Stetig(Z)^\bd \subseteq \mathcal{I}_Z$, 
  let a Hermitian element $f\in \Stetig(Z)^\bd_\Hermitian$ be given. Then one can construct an extension 
  $f^\mathrm{ext} \in \Stetig(\RR^N)^\bd_\Hermitian \subseteq \mathcal{I}_\RR$ fulfilling
  $f^\mathrm{ext} \at{Z} = f$ by applying Tietze's extension theorem, so $f \in \mathcal{I}_Z$.
  For general (not necessarily Hermitian) $f \in \Stetig(Z)^\bd$ this shows that $\RE(f), \IM(f) \in \mathcal{I}_Z$
  and therefore also $f \in \mathcal{I}_Z$.
  Note that instead of the most general form of Tietze's extension theorem, one can also use one
  of the well-known explicit formulae for such an extension of a continuous function $f \colon Z \to \RR$ which is bounded from below.
  An example is
  \begin{align}
    f^\mathrm{ext}(x)
    \coloneqq
    \inf_{z \in Z} \bigg( f(z) + \frac{\mathrm{d}(z,x) - \mathrm{d}_Z(x)}{\mathrm{d}_Z(x)} \bigg)
    \in \big[ {-\seminorm{\infty}{f}} , \seminorm{\infty}{f} \big]
    \label{eq:ucc:spec:internal2}   
    \tag{\textup{$**$}}
  \end{align}
  for all $x\in\RR^N\backslash Z$, where $\mathrm{d}$ is any metric on $\RR^N$ that induces its standard topology
  and $\mathrm{d}_Z \colon \RR^N \to {[0,\infty[}$ the distance from $Z$, i.e.~$\mathrm{d}_Z(x) \coloneqq \inf_{z\in Z} \mathrm{d}(z,x)$ for all $x\in \RR^N$.
  
  Like in \eqref{eq:ucc:Gamma} one can now construct a well-defined unital $^*$\=/homomorphism $\Phi_Z \colon \mathcal{I}_Z \to \mathcal{A}$,
  $\Phi_Z(f\at{Z}) \coloneqq \Phi_\RR(f)$,
  because $\ker \Phi_\RR \supseteq \set[\big]{f\in \mathcal{I}_\RR}{ f\at{Z} = 0}$ as shown above.
  Then $\Phi_Z(\pr_{Z;n}) = \Phi_\RR(\pr_{\RR^N;n}) = a_n$ holds for all $n\in \{1,\dots,N\}$,
  so $(Z,\mathcal{I}_Z,\Phi_Z)$ is a continuous calculus for $a_1,\dots,a_N$. It is even
  the universal one:
  Let any continuous calculus $(Y,\mathcal{J},\Psi)$ for $a_1,\dots,a_N$ be given. Then $Y \supseteq Z$
  because for every $x \in \RR^N \backslash Y$ there is an $\epsilon \in {]0,\infty[}$ such that $f_{x,\epsilon} \at Y \ge 0$
  as discussed above, and therefore $\sum_{n=1}^N (x_n\Unit_\mathcal{A}-a_n)^2 - \epsilon \Unit_\mathcal{A} = \Psi(f_{x,\epsilon} \at{Y}) \ge 0$,
  which shows that $x \in \RR^N \backslash Z$ by \eqref{eq:ucc:spec:internal}. It is now easy to check that $\big(\RR^N, \mathcal{J}^\mathrm{ext}, \Psi^\mathrm{ext}\big)$
  with $\mathcal{J}^\mathrm{ext} \coloneqq \set[\big]{f^\mathrm{ext}\in\Stetig(\RR^N)}{f^\mathrm{ext}\at{Y} \in \mathcal{J}}$ and
  $\Psi^\mathrm{ext}(f) \coloneqq \Psi(f^\mathrm{ext}\at{Y})$ for all $f^\mathrm{ext}\in \mathcal{J}^\mathrm{ext}$ is a continuous
  calculus for $a_1,\dots,a_N$, so Corollary~\ref{corollary:bigcalculus} shows that $\mathcal{J}^\mathrm{ext} \subseteq \mathcal{I}_\RR$
  and $\Psi^\mathrm{ext}(f^\mathrm{ext}) = \Phi_\RR(f^\mathrm{ext})$ for all $f^\mathrm{ext}\in \mathcal{J}^\mathrm{ext}$.
  Given any $f \in \mathcal{J}^+_\Hermitian$, then one can again construct an extension 
  $f^\mathrm{ext} \in \Stetig(\RR^N)^+_\Hermitian$ fulfilling $f^\mathrm{ext} \at{Y} = f$
  by using the explicit formula \eqref{eq:ucc:spec:internal2} with $Y$ in place of $Z$, or by taking the positive part of any extension of $f$ to $\RR^N$.
  But then
  $f^\mathrm{ext} \in \mathcal{J}^\mathrm{ext} \subseteq \mathcal{I}_\RR$ implies that $f\at{Z} = f^\mathrm{ext}\at{Z} \in \mathcal{I}_Z$
  and $\Phi_Z(f\at{Z}) = \Phi_\RR(f^\mathrm{ext}) = \Psi^\mathrm{ext}(f^\mathrm{ext}) = \Psi(f)$. As a general
  $f\in \mathcal{J}$ can be decomposed as $f = \sum_{k=0}^3 \I^k f_k$ with $f_k \in \mathcal{J}^+_\Hermitian$,
  e.g.~using the positive and negative parts of the real and imaginary parts of $f$, this guarantees
  that $f\at{Z} \in \mathcal{I}_Z$ and $\Phi_Z(f\at{Z}) = \Psi(f)$ hold for all $f\in \mathcal{J}$. So $(Z,\mathcal{I}_Z,\Phi_Z)$ is the
  universal continuous calculus for $a_1,\dots,a_N$, thus $\spec(a_1,\dots,a_N) = Z$, $\mathcal{F}(a_1,\dots,a_N) = \mathcal{I}_Z$
  and $\Gamma_{a_1,\dots,a_N} = \Phi_Z$, and \eqref{eq:ucc:spec}, \eqref{eq:ucc:F} and \eqref{eq:ucc:Gamma}
  are fulfilled.
  
  It only remains to check the additional properties mentioned in the statement of this theorem:
  By construction of $\Phi_Z$, its image in $\mathcal{A}$ is identical to the image of $\Phi_\RR$, hence is again a subset
  of the bicommutant $\{a_1,\dots,a_N\}''$, see Proposition~\ref{proposition:bigcalculus}.
  Being a unital $^*$\=/homomorphism defined on a proper \Sus\=/algebra of continuous functions, $\Phi_Z$
  is automatically positive by Corollary~\ref{corollary:niceMorphisms}. It is also injective because 
  $\Phi_Z(f\at{Z}) = 0$ for some $f\in \mathcal{I}_\RR$ means $\Phi_\RR(f) = 0$, i.e.~$f \in \ker \Phi_\RR$,
  which means $f\at{Z} = 0$ as discussed above. Corollary~\ref{corollary:niceMorphisms} now also shows that $\Phi_Z$ is 
  an order embedding, because if $\Phi_Z(f) \in \mathcal{A}^+_\Hermitian$ with some $f\in (\mathcal{I}_{Z})_\Hermitian$,
  then $\Phi_Z(f) = \Phi_Z(f_+)$ implies $f = f_+ \ge 0$. 
  
  As $\mathcal{I}_\RR$ is a commutative
  \Sus\=/algebra of finite type with generators $\pr_{\RR^N;1}, \dots, \pr_{\RR^N;N}$ by
  Proposition~\ref{proposition:bigcalculus}, and as $\mathcal{I}_Z$ is
  the image of $\mathcal{I}_\RR$ under the positive unital $^*$\=/homomorphism $\argument\at{Z}$,
  it is also a commutative \Sus\=/algebra of finite type with generators $\pr_{Z;1}, \dots, \pr_{Z;N}$
  by Proposition~\ref{proposition:surjectivefinite}.
  
  Finally, if $f \in \Stetig(Z)^+_\Hermitian$ is coercive, then there exists $\epsilon \in {]0,\infty[}$ such that $f(z) \ge \epsilon$
  for all $z\in Z$ and one can construct an extension $f^\mathrm{ext} \in \Stetig(\RR^N)^+_\Hermitian$ fulfilling $f^\mathrm{ext} \at{Z} = f$ and
  $f^\mathrm{ext}(x) \ge \epsilon$ for all $x\in \RR^N$ by using the explicit formula \eqref{eq:ucc:spec:internal2}
  or by taking the pointwise maximum with $\epsilon$ of any extension of $f$ to $\RR^N$. Its inverse 
  $(f^\mathrm{ext})^{-1} \in \Stetig(\RR^N)^\bd \subseteq \mathcal{I}_\RR$ then fulfils $(f^\mathrm{ext})^{-1} \at{Z} = f^{-1}$, thus 
  $\Phi_\RR\big((f^\mathrm{ext})^{-1} \big) = \Phi_Z(f^{-1})$.
  If $\Phi_Z(f^{-1})$ is invertible in $\mathcal{A}$, then it follows from the previous Lemma~\ref{lemma:largestpossible} that $f^\mathrm{ext} \in \mathcal{I}_\RR$
  and therefore $f \in \mathcal{I}_Z$. 
\end{proof}
Note that this allows to construct e.g.~absolute values and positive or negative parts
of Hermitian elements of a \Sus\=/algebra $\mathcal{A}$, and also square roots of positive Hermitian elements and
``suprema'' or ``infima'' of commuting Hermitian elements by application of the universal continuous
calculus to the corresponding function. The result is the same as the constructions
in \cite{schoetz:EquivalenceOrderAlgebraicStructure} because the algebraic and order theoretic properties match,
as one can easily check: For example, given $a\in \mathcal{A}^+_\Hermitian$, then $(x\Unit - a)^2 \ge x^2 \Unit$
for all $x\in {]{-\infty},0[}$ shows that $\spec(a) \subseteq {[0,\infty[}$, and $\Gamma_{a}\big(\sqrt{\argument}\at{\spec(a)}\big)$
yields the square root of $a$ because it is a positive Hermitian element of the bicommutant $\{a\}''$ that squares to $a$.

The universal continuous calculus applies especially to the \Sus\=/algebras of
operators on a Hilbert space as in \cite[Sec.~8]{schoetz:EquivalenceOrderAlgebraicStructure} 
that can be constructed out of any selfadjoint operator. The spectrum of a single Hermitian element
can then be described in the usual way:
\begin{corollary} \label{corollary:RRspec}
  Let $\mathcal{A}$ be a \Sus\=/algebra and $a\in \mathcal{A}_\Hermitian$, then
  \begin{equation}
    \RR\backslash\spec(a) 
    = 
    \set[\big]{\lambda \in \RR}{\lambda\Unit-a\text{ is invertible in }\mathcal{A}\text{ and }(\lambda\Unit-a)^{-1}\in\mathcal{A}^\bd}.
    \label{eq:RRspec}
  \end{equation}
\end{corollary}
\begin{proof}
  If $\lambda \in \RR$ has the property that $\lambda\Unit-a$ is invertible in $\mathcal{A}$
  with uniformly bounded inverse, then $(\lambda\Unit-a)^{-2} \le \seminorm{\infty}{(\lambda\Unit-a)^{-1}}^2 \Unit$
  implies $\seminorm{\infty}{(\lambda\Unit-a)^{-1}}^{-2} \,\Unit \le (\lambda\Unit-a)^2$, so $\lambda \in \RR\backslash \spec(a)$
  by the previous Theorem~\ref{theorem:univcc}.
  
  Conversely, let $\pr \coloneqq \id_\RR\at{\spec(a)}\in \mathcal{F}(a)$ be the coordinate function
  and let $\lambda \in \RR\backslash\spec(a)$ be given, then $\lambda\Unit_{\mathcal{F}(a)}-\pr$
  is invertible in $\Stetig(\spec(a))$ and its inverse is uniformly bounded because $\spec(a)$ is closed,
  so $(\lambda\Unit_{\mathcal{F}(a)}-\pr)^{-1} \in \Stetig(\spec(a))^\bd \subseteq \mathcal{F}(a)$.
  It follows that $\lambda\Unit_\mathcal{A}-a = \Gamma_a\big( \lambda\Unit_{\mathcal{F}(a)}-\pr\big)$ is invertible in $\mathcal{A}$
  with uniformly bounded inverse.
\end{proof}
A \neu{normal} element of a $^*$\=/algebra is an element $a$ fulfilling $a\,a^* = a^*a$. This is the case
if and only if its real and imaginary parts $\RE(a)$ and $\IM(a)$ commute. By identifying $\RR^2$ with $\CC$ via
the $\RR$-linear isomorphism $\RR^2 \ni (x_1,x_2) \mapsto x_1 + \I x_2 \in \CC$ and interpreting
$\spec_\CC(a) \coloneqq {\spec}\big(\RE(a),\IM(a)\big)$ as a closed subset of $\CC$, one obtains
the usual description of the spectrum of normal elements and the spectral mapping theorem:
\begin{corollary}
  Let $\mathcal{A}$ be a \Sus\=/algebra and $a\in \mathcal{A}$ a normal element, then
  \begin{equation}
    \CC\backslash\spec_\CC(a)
    = 
    \set[\big]{\lambda \in \CC}{\lambda\Unit-a\text{ is invertible in }\mathcal{A}\text{ and }(\lambda\Unit-a)^{-1}\in\mathcal{A}^\bd}.
    \label{eq:CCspec}
  \end{equation}
  Moreover, given pairwise commuting elements $a_1,\dots,a_N \in \mathcal{A}_\Hermitian$ with $N\in \NN$ and
  $f \in \mathcal{F}(a_1,\dots,a_N)$, then ${\spec_\CC}\big( \Gamma_{a_1,\dots,a_N}(f) \big)$
  is the closure in $\CC$ of the image of $\spec(a_1,\dots,a_N)$ under $f$.
\end{corollary}
\begin{proof}
  If $\lambda \in \CC$ has the property that $\lambda\Unit-a$ is invertible in $\mathcal{A}$
  with uniformly bounded inverse, then $((\lambda\Unit-a)^{-1})^*(\lambda\Unit-a)^{-1} \le \seminorm{\infty}{(\lambda\Unit-a)^{-1}}^2 \Unit$
  implies
  \begin{align*}
    \seminorm[\big]{\infty}{(\lambda\Unit-a)^{-1}}^{-2} \,\Unit
    \le
    (\lambda\Unit-a)^*(\lambda\Unit-a)
    =
    \big(\RE(\lambda)\Unit-\RE(a)\big)^2 + \big(\IM(\lambda)\Unit-\IM(a)\big)^2
    ,
  \end{align*}
  so $\lambda \in \CC\backslash \spec_\CC(a)$ by Theorem~\ref{theorem:univcc}.
  
  Conversely, $\id_\CC\at{\spec_\CC(a)} = \pr_1 + \I\, \pr_2 \in \mathcal{F}\big(\RE(a),\IM(a)\big)$,
  and for $\lambda \in \CC\backslash\spec_\CC(a)$ one thus finds that
  $\lambda \Unit - ( \pr_1 + \I\, \pr_2)$ is invertible in $\Stetig\big( \spec_\CC(a)\big)$
  with uniformly bounded inverse because $\spec_\CC(a)$ is closed, so 
  $\big(\lambda\Unit - (\pr_1 + \I\, \pr_2) \big){}^{-1} \in \Stetig\big( \spec_\CC(a)\big){}^\bd \subseteq \mathcal{F}\big(\RE(a),\IM(a)\big)$
  and $\lambda\Unit-a = \Gamma_{\RE(a),\IM(a)}\big( \lambda\Unit - (\pr_1 + \I\, \pr_2)\big)$ is invertible in $\mathcal{A}$
  with inverse $\Gamma_{\RE(a),\IM(a)}\big( (\lambda\Unit - (\pr_1 + \I\, \pr_2))^{-1}\big) \in \mathcal{A}^\bd$.
  
  Finally, let pairwise commuting $a_1,\dots,a_N \in \mathcal{A}_\Hermitian$ with $N\in \NN$ and $f\in \mathcal{F}(a_1,\dots,a_N)$
  be given. For $\lambda \in \CC$ define $f_\lambda \coloneqq \lambda \Unit -f\in \mathcal{F}(a_1,\dots,a_N)$,
  then for every $\epsilon \in {]0,\infty[}$, one has $f_\lambda^* f_\lambda \ge \epsilon^2 \Unit_{\mathcal{F}(a_1,\dots,a_N)}$ if and only
  if $\Gamma_{a_1,\dots,a_N}(f_\lambda^* f_\lambda) \ge \epsilon^2 \Unit_{\mathcal{A}}$ because $\Gamma_{a_1,\dots,a_N}$
  is an injective positive unital $^*$\=/homomorphism and an order embedding by Theorem~\ref{theorem:univcc}.
  As $\mathcal{A}$ is a \Sus\=/algebra and especially is symmetric, $\Gamma_{a_1,\dots,a_N}(f_\lambda^* f_\lambda) \ge \epsilon^2 \Unit_{\mathcal{A}}$
  is equivalent to $\Gamma_{a_1,\dots,a_N}(f_\lambda)$ being invertible in $\mathcal{A}$
  with inverse fulfilling $\big(\Gamma_{a_1,\dots,a_N}(f_\lambda)^{-1}\big)^*\Gamma_{a_1,\dots,a_N}(f_\lambda)^{-1} \le \epsilon^{-2} \Unit_\mathcal{A}$,
  i.e.~$\seminorm{\infty}{\Gamma_{a_1,\dots,a_N}(f_\lambda)^{-1}} \le \epsilon^{-1}$.
  It is now easy to check that, for fixed $\lambda \in \CC$, there exists $\epsilon \in {]0,\infty[}$
  such that $f_\lambda^* f_\lambda \ge \epsilon^2 \Unit_{\mathcal{F}(a_1,\dots,a_N)}$ if and only
  $\lambda$ is not an element of the closure of the image of $\spec(a_1,\dots,a_N)$ under $f$.
  By the above considerations, this is equivalent to 
  $\lambda \Unit_\mathcal{A} - \Gamma_{a_1,\dots,a_N}(f)$ being invertible in $\mathcal{A}$
  with inverse fulfilling $\seminorm{\infty}{(\lambda \Unit_\mathcal{A} - \Gamma_{a_1,\dots,a_N}(f))^{-1}} = \seminorm{\infty}{\Gamma_{a_1,\dots,a_N}(f_\lambda)^{-1}} \le \epsilon^{-1}$
  for this $\epsilon \in {]0,\infty[}$,
  i.e. to $\lambda \in \CC \backslash {\spec_\CC}\big( \Gamma_{a_1,\dots,a_N}(f) \big)$ by \eqref{eq:CCspec}.
\end{proof}
Note that the universal continuous calculus
for a normal element $a$ of a \Sus\=/algebra yields the expected results
$\Gamma_{\RE(a),\IM(a)}(\id_\CC \at{\spec_\CC(a)} ) = \Gamma_{\RE(a),\IM(a)}( \pr_1 + \I\, \pr_2) = a$
and $\Gamma_{\RE(a),\IM(a)}( \cc{\argument} \at{\spec_\CC(a)} ) = \Gamma_{\RE(a),\IM(a)}( \pr_1 - \I\, \pr_2) = a^*$.
\section{Representations of commutative \texorpdfstring{\Sus\=/algebras}{Su*-algebras} of finite type} \label{sec:rep}
Finally, the universal continuous calculus allows to identify the commutative \Sus\=/algebras of finite type
as the proper \Sus\=/algebras of continuous functions on closed subsets of Euclidean space:
\begin{proposition} \label{proposition:properonclosedsubsetofRN}
  Every proper \Sus\=/algebra of continuous functions on a closed subset $X$ of some $\RR^N$ with $N\in \NN$
  is a commutative \Sus\=/algebra of finite type.
\end{proposition}
\begin{proof}
  Let $X$ be a closed subset of $\RR^N$ for some $N\in \NN$ and let $\mathcal{I} \subseteq \Stetig(X)$ be a proper
  \Sus\=/algebra of continuous functions on $X$. We have to find generators for $\mathcal{I}$:
  
  First assume that $X$ is compact, then $\mathcal{I} = \Stetig(X) = \Stetig(X)^\bd$ and especially $\pr_n \in \mathcal{I}$ holds
  for all $n\in \{1,\dots,N\}$. In this case $\mathcal{I}$ is a commutative \Sus\=/algebra of finite type with generators
  $\pr_1,\dots,\pr_N$, because $p \coloneqq \Unit + \pr_1^2 + \dots + \pr_N^2$ is clearly proper in $\mathcal{I}$
  and because $\genCstar{\{p^{-1}, \pr_1 p^{-1}, \dots, \pr_N p^{-1}\}} = \mathcal{I}$ by the Stone-Weierstraß Theorem.
  
  Otherwise, i.e.~if $X$ is non-compact, it is possible that $\pr_n \notin \mathcal{I}$ for some $n\in \{1,\dots,N\}$ and the construction of generators
  is more complicated: In this case let $p \in \mathcal{I}^+_\Hermitian$ be a proper function and define
  $\seminorm{}{x} \coloneqq \max\big\{ \abs{x_1}, \dots,\abs{x_N} \big\}$ for all $x\in X$. Then for every $k\in \NN_0$
  there exists $x \in X$ with $\seminorm{}{x}\ge k$ and the sequence 
  $\NN_0 \ni k \mapsto \tilde{p}_k \coloneqq \inf_{x\in X, \seminorm{}{x}\ge k} p(x) \in {[0,\infty[}$
  is increasing and unbounded. One can now construct an unbounded sequence $(\alpha_k)_{k\in \NN_0}$
  that starts with $\alpha_0 = 0$ and which fulfils $\alpha_{k-1} < \alpha_{k}$ and $\alpha_k \le 1 + \tilde{p}_{k-1}$
  for all $k\in \NN$; for example, $\alpha_k \coloneqq 1-2^{-k} + \tilde{p}_{k-1}$ for all $k\in \NN$ is a suitable choice.
  From this sequence $(\alpha_k)_{k\in \NN}$ one obtains a piecewise defined continuous function $g \colon \RR \to \RR$,
  \begin{align*}
    t \mapsto g(t) \coloneqq
              \begin{cases}
                \alpha_{k-1}(k-t) + \alpha_k(1-k+t) & \text{ if } t \in {[k-1,k]} \text{ for some }k \in \NN\\
                -\alpha_{k-1}(k+t)-\alpha_k (1-k-t)  & \text{ if } t \in {[{-k},{-k+1}]} \text{ for some }k \in \NN
              \end{cases}
  \end{align*}
  fulfilling $g(-t) = - g(t)$ for all $t\in \RR$ and 
  $\abs{g(x_n)} = g(\abs{x_n}) \le \alpha_{k} \le 1+\tilde{p}_{k-1} \le 1+p(x)$ for all $x\in X$
  and all $n\in \{1,\dots,N\}$ and for a suitable $k\in \NN$ such that $k-1 \le \abs{x_n} \le k$.
  This function $g$
  is strictly increasing, i.e.~$g(t) < g(t')$ for all $t,t' \in \RR$ with $t< t'$, hence injective
  and open. Moreover, $\lim_{t\to\pm \infty} g(t) = \pm \infty$ holds, so $g$ is also surjective and therefore
  is a homeomorphism. As a consequence, the map
  $g^{\times N} \colon \RR^N \to \RR^N$, $(x_1,\dots,x_N) \mapsto \big(g(x_1), \dots,g(x_N)\big)$ is a homeomorphism, too.
  Now define $Y$ as the image of $X$ under $g^{\times N}$, which is a closed subset
  of $\RR^N$ because $g^{\times N}$ is a homeomorphism. Define
  $\phi \coloneqq X \to Y$ as $\phi(x) \coloneqq g^{\times N}(x)$ for all $x\in X$, which is a well-defined homeomorphism,
  and $\mathcal{J} \coloneqq \set[\big]{f\in\Stetig(Y)}{f\circ \phi \in \mathcal{I}}$. Then it is easy to check
  that $\mathcal{J}$ is an intermediate $^*$\=/subalgebra of $\Stetig(Y)$ and that the map 
  $\Phi \colon \mathcal{J} \to \mathcal{I}$, $f \mapsto \Phi(f) \coloneqq f\circ \phi$
  is an isomorphism of ordered $^*$\=/algebras whose inverse is given by $\mathcal{I} \ni f\mapsto f\circ \phi^{-1} \in \mathcal{J}$.
  Moreover, $\pr_{Y;n} \in \mathcal{J}$ for all $n\in \{1,\dots,N\}$
  because $\abs{\pr_{Y;n} \circ \phi} = \abs{g \circ\pr_{X;n}} \le \Unit_\mathcal{I} + p$ shows that $\pr_{Y;n} \circ \phi \in \mathcal{I}$  by Proposition~\ref{proposition:intermediateSus}.
  So $(Y,\mathcal{J},\Phi)$ is a continuous calculus for $g \circ\pr_{X;1}, \dots, g \circ\pr_{X;N} \in \mathcal{I}_\Hermitian$
  and $\Phi$ maps surjectively onto $\mathcal{I}$. It follows that the universal continuous
  calculus for $g \circ\pr_{X;1}, \dots, g \circ\pr_{X;N}$ also maps surjectively onto $\mathcal{I}$,
  hence is an isomorphism (one actually finds that $(Y,\mathcal{J},\Phi)$ is the universal continuous calculus
  because $\Phi$ is an isomorphism). 
  By Theorem~\ref{theorem:univcc}, $\mathcal{I} \cong \mathcal{F}( g \circ\pr_{X;1}, \dots, g \circ\pr_{X;N} )$
  is a commutative \Sus\=/algebra of finite type and $g \circ\pr_{X;1}, \dots, g \circ\pr_{X;N}$ are generators of $\mathcal{I}$.
\end{proof}
The converse is also true:
\begin{theorem} \label{theorem:representation}
  For every commutative \Sus\=/algebra of finite type $\mathcal{A}$ with
  generators $a_1,\dots,a_N \in \mathcal{A}_\Hermitian$, $N\in \NN$, the universal continuous calculus
  $\Gamma_{a_1,\dots,a_N} \colon \mathcal{F}(a_1,\dots,a_N) \to \mathcal{A}$ is an isomorphism of ordered $^*$\=/algebras,
  and thus allows to identify $\mathcal{A}$ with the proper \Sus\=/algebra of continuous functions $\mathcal{F}(a_1,\dots,a_N)$
  on the closed subset $\spec(a_1,\dots,a_N)$ of $\RR^N$.
\end{theorem}
\begin{proof}
  Let $\mathcal{A}$ be a commutative \Sus\=/algebra of finite type with generators $a_1,\dots,a_N \in \mathcal{A}_\Hermitian$, $N\in \NN$,
  so $p \coloneqq \Unit + a_1^2 + \dots + a_N^2$ is proper in $\mathcal{A}$. By Theorem~\ref{theorem:univcc}, the universal continuous 
  calculus for $a_1,\dots,a_N$ exists and $\Gamma_{a_1,\dots,a_N} \colon \mathcal{F}(a_1,\dots,a_N) \to \mathcal{A}$
  is an injective unital $^*$\=/homomorphism and even an order embedding. We have to show that $\Gamma_{a_1,\dots,a_N}$
  is also surjective:
  
  Given $b\in \mathcal{A}$, then there are numerator and denominator
  $b_\mathrm{n}',b_\mathrm{d}' \in \genCstar{\{p^{-1}, a_1 p^{-1}, \dots, a_N p^{-1}\}}$
  such that $b_\mathrm{d}'$ is invertible in $\mathcal{A}$ and $b = b_\mathrm{n}' (b'_\mathrm{d})^{-1}$.
  By setting $b_\mathrm{n} \coloneqq b_\mathrm{n}' (b_\mathrm{d}')^*, b_\mathrm{d} \coloneqq b_\mathrm{d}'(b_\mathrm{d}')^*$
  one obtains numerator and denominator
  $b_\mathrm{n},b_\mathrm{d} \in \genCstar{\{p^{-1}, a_1 p^{-1}, \dots, a_N p^{-1}\}}$
  such that $b_\mathrm{d}$ is positive Hermitian and invertible in $\mathcal{A}$ and $b = b_\mathrm{n} b_\mathrm{d}^{-1}$.
  By Proposition~\ref{proposition:sphericalcalculus} and the subsequent construction of the
  universal continuous calculus in Proposition~\ref{proposition:bigcalculus} and Theorem~\ref{theorem:univcc},
  the $C^*$\=/subalgebra $\genCstar{\{p^{-1}, a_1 p^{-1}, \dots, a_N p^{-1}\}}$
  of $\mathcal{A}^\bd$ lies in the image of $\Gamma_{a_1,\dots,a_N}$, so there especially exist
  $f_\mathrm{n},f_\mathrm{d} \in \mathcal{F}(a_1,\dots,a_N)$ such that $\Gamma_{a_1,\dots,a_N}(f_\mathrm{n}) = b_\mathrm{n}$
  and $\Gamma_{a_1,\dots,a_N}(f_\mathrm{d}) = b_\mathrm{d}$. As $\Gamma_{a_1,\dots,a_N}$ is
  injective and an order embedding, $f_\mathrm{d}$ is positive Hermitian and uniformly bounded.
  Moreover, even $f_\mathrm{d} (x) > 0$ holds for all $x\in \spec(a_1,\dots,a_N)$:
  
  Assume to the contrary that there is some $\hat{x}\in \spec(a_1,\dots,a_N)$ for which $f_\mathrm{d} (\hat{x}) = 0$.
  In this case, let $\tilde{p} \coloneqq \Unit + \pr_1^2 + \dots + \pr_N^2 \in \mathcal{F}(a_1,\dots,a_N)^+_\Hermitian$
  and $\lambda \coloneqq \tilde{p}(\hat{x}) + 1 \in {[1,\infty[}$, then
  $\Gamma_{a_1,\dots,a_N}(\tilde{p}) = p$ is proper in $\mathcal{A}$ and therefore
  there are $\mu \in {[0,\infty[}$ and $c\in \mathcal{A}^+_\Hermitian$ such that for every $\epsilon \in {]0,\infty[}$
  there is a $k\in \NN_0$ for which $b_{\mathrm{d}}^{-1} \le \mu \Unit_\mathcal{A} + (p/\lambda)^k + \epsilon c$ holds.
  Construct $g \in \mathcal{F}(a_1,\dots,a_N)_\Hermitian^+$ as the positive part of
  $\frac{1}{2} \big( \Unit - 2(\mu+1) f_{\mathrm{d}} - \tilde{p}/\lambda \big)$, 
  then $0 \le g \le \Unit /2$, $g(\hat{x}) = (2\lambda)^{-1} > 0$ and $g(x) = 0$ holds for all $x\in \spec(a_1,\dots,a_N)$ for which
  $\tilde{p}(x) \ge \lambda$ and also for all $x\in \spec(a_1,\dots,a_N)$ for which $(\mu+1) f_{\mathrm{d}}(x) \ge 1/2$.
  Multiplying the inequality $b_{\mathrm{d}}^{-1} \le \mu \Unit + (p/\lambda)^k + \epsilon c$ with $b_{\mathrm{d}} = (b_\mathrm{d}')^* b_\mathrm{d}'$
  yields $\Unit \le \big(\mu \Unit + (p/\lambda)^k\big)b_{\mathrm{d}} + \epsilon b_{\mathrm{d}} c$, and as $\Unit$, $b_{\mathrm{d}}$
  and $p$ have (necessarily unique) preimages under $\Gamma_{a_1,\dots,a_N}$, one obtains the estimate
  \begin{align*}
    \Gamma_{a_1,\dots,a_N} \Big( \Unit - \big(\mu \Unit + (\tilde{p}/\lambda)^k\big)f_{\mathrm{d}} \Big) \le \epsilon b_{\mathrm{d}} c
  \end{align*}
  and therefore also
  \begin{align*}
    \Gamma_{a_1,\dots,a_N} (g^2)
    \le
    \Gamma_{a_1,\dots,a_N}(\sqrt{g}) \,\Gamma_{a_1,\dots,a_N}\Big( \Unit - \big(\mu \Unit + (\tilde{p}/\lambda)^k\big) f_{\mathrm{d}} \Big) \, \Gamma_{a_1,\dots,a_N}(\sqrt{g})
    \le
    \epsilon b_{\mathrm{d}} c \Gamma_{a_1,\dots,a_N}(g)
    ,
  \end{align*}
  where the first inequality follows from the estimate
  $g^2 \le \big( \Unit - \big(\mu \Unit + (\tilde{p}/\lambda)^k\big) f_{\mathrm{d}}\big)g$
  that can easily be checked pointwise for $x\in \spec(a_1,\dots,a_N)$ 
  by using that $g(x) \neq 0$ only if $\big( \Unit - f_{\mathrm{d}}\big(\mu \Unit + (\tilde{p}/\lambda)^k\big) \big)(x) > 1/2$.
  As $g$ is independent of $k$, this estimate $\Gamma_{a_1,\dots,a_N} (g^2) \le \epsilon b_{\mathrm{d}} c \Gamma_{a_1,\dots,a_N}(g)$
  holds for all $\epsilon \in {]0,\infty[}$, so $\Gamma_{a_1,\dots,a_N} (g^2) = 0$ because $\mathcal{A}$ is Archimedean,
  therefore $g^2 = 0$ because $\Gamma_{a_1,\dots,a_N}$ is injective. But this contradicts $g(\hat{x}) > 0$, so the assumption
  that there exists $\hat{x} \in \spec(a_1,\dots,a_N)$ with $f_{\mathrm{d}}(\hat{x}) = 0$ has to be false.
  
  We thus have seen that $f_\mathrm{d} (x) > 0$ holds for all $x\in \spec(a_1,\dots,a_N)$, which means that the pointwise inverse
  $f_\mathrm{d}^{-1} \in \Stetig\big( \spec(a_1,\dots,a_N) \big)^+_\Hermitian$ exists. Even $f_\mathrm{d}^{-1} \in \mathcal{F}(a_1,\dots,a_N)$
  holds by the criterium from Theorem~\ref{theorem:univcc} because the inverse $f_\mathrm{d}$ of $f_\mathrm{d}^{-1}$ is uniformly
  bounded and $\Gamma_{a_1,\dots,a_N}(f_\mathrm{d}) = b_\mathrm{d}$ invertible in $\mathcal{A}$. It follows that $\Gamma_{a_1,\dots,a_N}(f_\mathrm{n}f_\mathrm{d}^{-1}) = b_\mathrm{n} b_\mathrm{d}^{-1} = b$
  and we conclude that $\Gamma_{a_1,\dots,a_N} \colon \mathcal{F}(a_1,\dots,a_N) \to \mathcal{A}$ is surjective, hence an isomorphism of ordered $^*$\=/algebras.
\end{proof}
In the above proof, the assumption of properness enters in an essential, but perhaps not very intuitiv way. A short example therefore might be helpful:
\begin{example}
  In the commutative \Sus\=/algebra $\Stetig\big({]0,1]}\big)$ the function $\id_{]0,1]} \in \Stetig\big({]0,1]}\big){}^\bd$ alone is not a sufficient tuple of generators
  because $\spec(\id_{]0,1]}) = {[0,1]}$ by Corollary~\ref{corollary:RRspec}, so that $\mathcal{F}(\id_{]0,1]}) = \Stetig\big([0,1]\big) = \Stetig\big([0,1]\big){}^\bd$
  is uniformly bounded and its image under $\Gamma_{\id_{]0,1]}}$ is only a unital $^*$\=/subalgebra of $\Stetig\big({]0,1]}\big){}^\bd$.
  While one can express every $f\in \Stetig\big({]0,1]}\big)$ as a quotient $f = f_{\mathrm{n}} f_{\mathrm{d}}^{-1}$
  with $f_{\mathrm{n}}, f_{\mathrm{d}}$ in the image of $\Gamma_{\id_{]0,1]}}$ and $f_{\mathrm{d}}$ invertible in $\Stetig\big({]0,1]}\big)$,
  e.g.~$f_{\mathrm{d}} = ( f^*f+(\id_{]0,1]})^{-1})^{-1}$ and $f_{\mathrm{n}} = f ( f^*f+(\id_{]0,1]})^{-1})^{-1}$,
  the preimage of such an invertible $f_{\mathrm{d}}$ under $\Gamma_{\id_{]0,1]}}$ need not be invertible:
  Indeed, $\id_{[0,1]} \in \mathcal{F}(\id_{]0,1]})$ is the preimage of $\id_{]0,1]} \in \Stetig\big({]0,1]}\big)$.
  This all relates to $\Unit+(\id_{]0,1]})^2 \in \Stetig\big({]0,1]}\big)$ not being proper.
\end{example}
As a final remark we note that Proposition~\ref{proposition:niceMorphisms} is especially interesting in connection with the above Theorem~\ref{theorem:representation}:

Let $\mathcal{A}$ and $\mathcal{B}$ be two commutative \Sus\=/algebras of finite type with generators $a_1,\dots,a_M \in \mathcal{A}_\Hermitian$
and $b_1,\dots,b_N \in \mathcal{B}_\Hermitian$, $M,N\in \NN$, which thus are isomorphic as ordered $^*$\=/algebras
to the proper \Sus\=/algebras of continuous functions $\mathcal{F}(a_1,\dots,a_M)$ and $\mathcal{F}(b_1,\dots,b_N)$
via the isomorphisms $\Gamma_{a_1,\dots,a_M}$ and $\Gamma_{b_1,\dots,b_N}$, respectively.
Then for every unital $^*$\=/homomorphism $\Phi \colon \mathcal{A} \to \mathcal{B}$ there exists a unique continuous map 
$\phi \colon \spec(b_1,\dots,b_N) \to \spec(a_1,\dots,a_M)$ such that $f\circ \phi \in \mathcal{F}(b_1,\dots,b_N)$ and 
$\Phi\big( \Gamma_{a_1,\dots,a_M}(f)\big) = \Gamma_{b_1,\dots,b_N}(f\circ \phi)$ hold for all $f\in\mathcal{F}(a_1,\dots,a_M)$.

\end{onehalfspace}
\end{document}